\documentclass{amsart}

\usepackage{amssymb,amsmath,amsthm,latexsym,booktabs}

\theoremstyle{definition}
\newtheorem{definition}{Definition}[section]
\newtheorem{remark}[definition]{Remark}
\newtheorem{example}[definition]{Example}

\theoremstyle{plain}
\newtheorem{lemma}[definition]{Lemma}
\newtheorem{theorem}[definition]{Theorem}

\begin{document}

\title[Every $3 \times 3 \times 3$ array over $\mathbb{C}$ has rank $\le 5$]
{On Kruskal's theorem that every $\mathbf{3 \times 3 \times 3}$ array has rank at most 5}

\author{Murray R. Bremner}

\address{Department of Mathematics and Statistics, University of Saskatchewan, Canada}

\email{bremner@math.usask.ca}

\author{Jiaxiong Hu}

\address{Department of Mathematics, Simon Fraser University, Canada}

\email{hujiaxiong@gmail.com}

\begin{abstract}
In the first part of this paper, we consider $3 \times 3 \times 3$ arrays with complex entries,
and provide a complete self-contained proof of Kruskal's theorem that the maximum rank is 5.
In the second part, we provide a complete classification of the canonical forms of $3 \times 3 \times 3$ 
arrays over $\mathbb{F}_2$; in particular, we obtain explicit examples of such arrays with rank 6.
\end{abstract}



\maketitle


In 1989, Kruskal \cite[page 10]{Kruskal} stated without proof that every $3 \times 3 \times 3$ array with real entries has rank at most 5.
A few years later, Rocci \cite{Rocci} circulated a simplified proof of this result, based on Kruskal's unpublished hand-written notes.
The details of this argument appear never to have been published.
In sections \ref{tenBergesection}--\ref{Kruskalsection}, we consider $3 \times 3 \times 3$ arrays with complex entries,
and provide a complete self-contained proof that the maximum rank is 5.

In section \ref{F2section} we consider this problem over the field $\mathbb{F}_2$ with two elements.
A remarkable fact, first noted by von zur Gathen \cite{vonzurGathen}, is that in this case
there exist $3 \times 3 \times 3$ arrays of rank 6.
We provide a complete classification of the canonical forms of $3 \times 3 \times 3$ arrays over $\mathbb{F}_2$;
in particular, we obtain explicit examples of such arrays with rank 6.

We use without reference many basic results on multidimensional arrays which can be found in
de Silva and Lim \cite{deSilvaLim} and Kolda and Bader \cite{KoldaBader}.


\section{Preliminaries on 3-dimensional arrays} \label{preliminaries}

We consider a $p \times q \times r$ array $X$ with entries in an arbitrary field $\mathbb{F}$ of scalars:
  \[
  X = [ \, x_{ijk} \, ],
  \qquad
  x_{ijk} \in \mathbb{F},
  \qquad
  1 \le i \le p, \qquad
  1 \le j \le q, \qquad
  1 \le k \le r.
  \]
By a slice of $X$ we mean any (2-dimensional) submatrix obtained by fixing one index.
Fixing $i$ gives a horizontal slice, fixing $j$ gives a vertical slice, and fixing $k$ gives a frontal slice.
The matrix form of $X$ is the $p \times qr$ matrix obtained by concatenating the frontal slices
$X_1, \dots, X_r$ from left to right:
  \[
  X
  =
  \left[
  \begin{array}{c|c|c}
  X_1 & \cdots & X_r
  \end{array}
  \right]
  =
  \left[
  \begin{array}{ccc|c|ccc}
  x_{111} & \cdots & x_{1q1} & \quad \cdots \quad & x_{11r} & \cdots & x_{1qr} \\
  \vdots & \ddots & \vdots & \quad \cdots \quad & \vdots & \ddots & \cdots \\
  x_{p11} & \cdots & x_{pq1} & \quad \cdots \quad & x_{p1r} & \cdots & x_{pqr}
  \end{array}
  \right].
  \]
Given three column vectors,
  \[
  \mathbf{a} = \begin{bmatrix} a_1 \\ \vdots \\ a_p \end{bmatrix} \in \mathbb{F}^p,
  \qquad
  \mathbf{b} = \begin{bmatrix} b_1 \\ \vdots \\ b_q \end{bmatrix} \in \mathbb{F}^q,
  \qquad
  \mathbf{c} = \begin{bmatrix} c_1 \\ \vdots \\ c_r \end{bmatrix} \in \mathbb{F}^r,
  \]
their outer product $\mathbf{a} \otimes \mathbf{b} \otimes \mathbf{c}$ is the $p \times q \times r$ array whose
$ijk$ entry is $a_i b_j c_k$.
A simple tensor is an outer product of nonzero vectors.
A fundamental problem is to represent the array $X$ as a sum of simple tensors:
  \[
  X = \sum_{i=1}^n \mathbf{a}^{(i)} \otimes \mathbf{b}^{(i)} \otimes \mathbf{c}^{(i)}.
  \]
The rank of the array $X$ is the smallest non-negative integer $n$ for which this decomposition is possible.
The rank is 0 if and only if every entry of the array is 0; the rank is
1 if and only if the array is a simple tensor.

The rank does not change if we permute the slices in each direction.
Given permutations $\alpha \in S_p$, $\beta \in S_q$, $\gamma \in S_r$, we form another $p \times q \times r$ array by
  \[
  \big( ( \alpha, \beta, \gamma ) \cdot X \big)_{ijk}
  =
  x_{\alpha(i)\beta(j)\gamma(k)}.
  \]
More generally, the rank does not change if we apply a change of basis in each direction.
Given invertible matrices
  \[
  A = ( a_{i_1i_2} ) \in GL(p,\mathbb{F}),
  \qquad
  B = ( b_{j_1j_2} ) \in GL(q,\mathbb{F}),
  \qquad
  C = ( c_{k_1k_2} ) \in GL(r,\mathbb{F}),
  \]
we form another $p \times q \times r$ array by
  \[
  \big( ( A, B, C ) \cdot X \big)_{i_1j_1k_1}
  =
  \sum_{i_2=1}^p \sum_{j_2=1}^q \sum_{k_2=1}^r
  a_{i_1i_2} b_{j_1j_2} c_{k_1k_2}
  x_{i_2j_2k_2}.
  \]
The rank does not change if we permute the directions; however, this permutes the dimensions $p, q, r$
and hence may give a different ordered triple $(p,q,r)$.
If we write the dimensions as $p_1 \times p_2 \times p_3$ with corresponding indices $i_1, i_2, i_3$
then applying a permutation $\delta \in S_3$ gives an array of size $p_{\delta(1)} \times p_{\delta(2)} \times p_{\delta(3)}$ defined by
  \[
  ( \delta \cdot X )_{i_{\delta(1)}i_{\delta(2)}i_{\delta(3)}} = x_{i_1i_2i_3}.
  \]
In the rest of this paper, we often use these rank-preserving transformations without further comment.


\section{ten Berge's theorem on $2 \times 2 \times 2$ arrays} \label{tenBergesection}

The results in this section are taken with minor changes from ten Berge \cite{tenBerge}.
However, for us the base field is $\mathbb{C}$ whereas for ten Berge it is $\mathbb{R}$.
We recall these results in detail since they are essential to the analysis of $3 \times 3 \times 3$ arrays.
For $2 \times 2 \times 2$ arrays, the rank decomposition takes the form
  \[
  X = \sum_{i=1}^n \mathbf{a}^{(i)} \otimes \mathbf{b}^{(i)} \otimes \mathbf{c}^{(i)},
  \quad \text{where} \quad
  \mathbf{a}^{(i)}, \mathbf{b}^{(i)}, \mathbf{c}^{(i)} \in \mathbb{C}^2
  \; \text{for} \;
  1 \le i \le n.
  \]
We express this decomposition in terms of three $2 \times n$ matrices $A, B, C$:
  \[
  A
  =
  \begin{bmatrix}
  \, \mathbf{a}^{(1)} & \cdots & \mathbf{a}^{(n)}
  \end{bmatrix},
  \quad
  B
  =
  \begin{bmatrix}
  \, \mathbf{b}^{(1)} & \cdots & \mathbf{b}^{(n)}
  \end{bmatrix},
  \quad
  C
  =
  \begin{bmatrix}
  \, \mathbf{c}^{(1)} & \cdots & \mathbf{c}^{(n)}
  \end{bmatrix}.
  \]

\begin{lemma} \label{lemma1} \cite[p.~632]{tenBerge}
The rank of a nonzero $2 \times 2 \times 2$ array $X$ is the least integer $n \ge 1$ such that
the frontal slices $X_1$, $X_2$ have the form $X_1 = A D B^t$, $X_2 = A E B^t$
where $A$, $B$ are $2 \times n$ matrices and $D$, $E$ are $n \times n$ diagonal matrices.
\end{lemma}

\begin{proof}
The first frontal slice $X_1$ has the form
  \allowdisplaybreaks
  \begin{align*}
  X_1
  &=
  \sum_{i=1}^n
  \left[
  \begin{array}{cc}
  a^{(i)}_1 b^{(i)}_1 c^{(i)}_1 & a^{(i)}_1 b^{(i)}_2 c^{(i)}_1 \\[3pt]
  a^{(i)}_2 b^{(i)}_1 c^{(i)}_1 & a^{(i)}_2 b^{(i)}_2 c^{(i)}_1
  \end{array}
  \right]
  =
  \sum_{i=1}^n
  \left[
  \begin{array}{cc}
  A_{1i} \, c^{(i)}_1 B^t_{i1} & A_{1i} \, c^{(i)}_1 B^t_{i2} \\[3pt]
  A_{2i} \, c^{(i)}_1 B^t_{i1} & A_{2i} \, c^{(i)}_1 B^t_{i2}
  \end{array}
  \right]
  \\
  &=
  \left[
  \begin{array}{cc}
  \sum_{i=1}^n A_{1i} \, c^{(i)}_1 B^t_{i1} & \sum_{i=1}^n A_{1i} \, c^{(i)}_1 B^t_{i2} \\[3pt]
  \sum_{i=1}^n A_{2i} \, c^{(i)}_1 B^t_{i1} & \sum_{i=1}^n A_{2i} \, c^{(i)}_1 B^t_{i2}
  \end{array}
  \right]
  =
  A \, C_1 B^t,
  \end{align*}
where $C_1$ is the $n \times n$ diagonal matrix whose diagonal entries $c^{(1)}_1$, $c^{(2)}_1$, $\dots$, $c^{(n)}_1$ come from row 1 of $C$.
Similarly, for the second frontal slice we have $X_2 = A \, C_2 B^t$, where $C_2$ is the $n \times n$ diagonal matrix whose diagonal entries
come from the second row of $C$.
Conversely, if the two frontal slices $X_1$ and $X_2$ can be written as $A \, C_1 B^t$ and $A \, C_2 B^t$ where $A$ and $B$
are $2 \times n$ matrices and $C_1$ and $C_2$ are $n \times n$ diagonal matrices, then $X$ has the given decomposition.
\end{proof}

\begin{definition}
We call the $2 \times 2 \times 2$ array $X$ \textbf{superdiagonal} if it has one of these forms for $\alpha, \beta \in \mathbb{C} \setminus \{0\}$:
  \[
  \left[
  \begin{array}{cc|cc}
  \alpha & 0 & 0 & 0 \\
  0 & 0 & 0 & \beta
  \end{array}
  \right],
  \quad
  \left[
  \begin{array}{cc|cc}
  0 & \alpha & 0 & 0 \\
  0 & 0 & \beta & 0
  \end{array}
  \right],
  \quad
  \left[
  \begin{array}{cc|cc}
  0 & 0 & 0 & \beta \\
  \alpha & 0 & 0 & 0
  \end{array}
  \right],
  \quad
  \left[
  \begin{array}{cc|cc}
  0 & 0 & \beta & 0 \\
  0 & \alpha & 0 & 0
  \end{array}
  \right].
  \]
\end{definition}

\begin{lemma} \label{lemma2} \cite[p.~632]{tenBerge}
A superdiagonal array has rank 2.
\end{lemma}

\begin{proof}
By applying permutations of the slices, we may assume that $X$ has the first form.
It is then clear that the array has rank $\le 2$ since
  \[
  \left[
  \begin{array}{cc|cc}
  \alpha & 0 & 0 & 0 \\
  0 & 0 & 0 & \beta
  \end{array}
  \right]
  =
  \left[ \begin{array}{c} \alpha \\ 0 \end{array} \right]
  \otimes
  \left[ \begin{array}{c} 1 \\ 0 \end{array} \right]
  \otimes
  \left[ \begin{array}{c} 1 \\ 0 \end{array} \right]
  +
  \left[ \begin{array}{c} 0 \\ \beta \end{array} \right]
  \otimes
  \left[ \begin{array}{c} 0 \\ 1 \end{array} \right]
  \otimes
  \left[ \begin{array}{c} 0 \\ 1 \end{array} \right].
  \]
To find the general form of an array of rank 1 according to Lemma \ref{lemma1}, we set
  \[
  A = \left[ \begin{array}{c} a_1 \\ a_2 \end{array} \right],
  \qquad
  B = \left[ \begin{array}{c} b_1 \\ b_2 \end{array} \right],
  \qquad
  D = \left[ \begin{array}{c} d \end{array} \right],
  \qquad
  E = \left[ \begin{array}{c} e \end{array} \right].
  \]
We obtain
  \[
  X_1
  =
  A D B^t
  =
  \left[
  \begin{array}{cc}
  a_1 d b_1 & a_1 d b_2 \\
  a_2 d b_1 & a_2 d b_2
  \end{array}
  \right],
  \qquad
  X_2
  =
  A E B^t
  =
  \left[
  \begin{array}{cc}
  a_1 e b_1 & a_1 e b_2 \\
  a_2 e b_1 & a_2 e b_2
  \end{array}
  \right],
  \]
or more simply $X_1 = d(AB^t)$ and $X_2 = e(AB^t)$.
Thus $X_1$ and $X_2$ are scalar multiples of the same matrix of rank 1.
This does not hold for a superdiagonal array, which therefore has rank $\ge 2$.
\end{proof}

\begin{lemma} \label{lemma3} \cite[p.~632]{tenBerge}
Let $X$ be a nonzero $2 \times 2 \times 2$ array which is not superdiagonal.
Then $X$ has rank 1 if and only if all six of its slices are singular.
\end{lemma}

\begin{proof}
($\Rightarrow$)
We show that if $X$ has a non-singular slice, then its rank is $\ge 2$.
By permuting the directions, we may assume that a frontal slice is non-singular.
By permuting the frontal slices, we may assume that $X_1$ is non-singular.
If the rank of $X$ is 1 then as in the proof of Lemma \ref{lemma2} we have
$X_1 = d( \mathbf{a} \otimes \mathbf{b} )$ where $\mathbf{a}$ and $\mathbf{b}$ are nonzero vectors in $\mathbb{C}^2$;
but this matrix is clearly singular.

($\Leftarrow$)
We show that if all six slices are singular then $X$ has rank 1.

Case 1: Some slice is zero; by permuting the directions and slices we may assume that $X_1 = 0$.
Since $X_2$ is nonzero and singular we have $X_2 = \mathbf{a} \otimes \mathbf{b}$
for some nonzero vectors $\mathbf{a}, \mathbf{b} \in \mathbb{C}^2$.
But then $X = \mathbf{a} \otimes \mathbf{b} \otimes \mathbf{e}_1$
where $\mathbf{e}_1 = [0,1]^t$.

Case 2: No slice is zero.
Since $X_1$ is nonzero and singular, we have $X_1 = \mathbf{a} \otimes \mathbf{b}$ where $\mathbf{a} = [ a_1, a_2 ]^t$
and $\mathbf{b} = [b_1,b_2]^t$ are nonzero vectors.
By transposing the vertical slices of $X$ if necessary, we may assume that the first column of $X_1$ is nonzero.
Equivalently, $\mathbf{b} = [ 1, \lambda ]^t$ for some $\lambda \in \mathbb{C}$; thus $X_1 = [ \mathbf{a} | \lambda \mathbf{a} ]$.

Subcase 2(a): $\lambda = 0$.
Since the first vertical slice is singular,
$X = [ \mathbf{a}, \mathbf{0} | \mu \mathbf{a} , \mathbf{d} ]$
for some $\mu \in \mathbb{C}$ and some $\mathbf{d}$; we have $\mathbf{d} \ne \mathbf{0}$ since the second vertical slice is nonzero.
If $\mu \ne 0$ then since $X_2$ is singular, there is $\nu \in \mathbb{C} \setminus \{0\}$ such that
$X = [ \mathbf{a} , \mathbf{0} | \mu \mathbf{a}, \nu \mathbf{a} ]$.
In this case, since the horizontal slices are nonzero, we have $a_1 \ne 0$, $a_2 \ne 0$.
Since the horizontal slices are singular, it follows that $\mu = 0$, $\nu = 0$.  But then the
second frontal slice is zero, giving a contradiction.
If $\mu = 0$ then $X = [ \mathbf{a}, \mathbf{0} | \mathbf{0}, \mathbf{d} ]$.
In this case, since the horizontal slices are nonzero and singular, $X$
must be a superdiagonal array, again giving a contradiction.

Subcase 2(b): $\lambda \ne 0$.  We have $X = [ \mathbf{a}, \lambda \mathbf{a} | \mu \mathbf{a}, \mathbf{d} ]$.
But $\mathbf{a} \ne \mathbf{0}$ and the second vertical slice is singular, so $\mathbf{d} = \nu \mathbf{a}$ for some $\nu \in \mathbb{C}$,
giving $X = [ \mathbf{a}, \lambda \mathbf{a} | \mu \mathbf{a}, \nu \mathbf{a} ]$.
Since either $a_1 \ne 0$ or $a_2 \ne 0$ (or both), singularity of the horizontal slices implies that $\nu = \lambda \mu$.
Then $X = [ \mathbf{a}, \lambda \mathbf{a} | \mu \mathbf{a}, \lambda \mu \mathbf{a} ]
  =
  [ a_1, a_2 ]^t
  \otimes
  [ 1, \lambda ]^t
  \otimes
  [ 1, \mu ]^t$
has rank 1.
\end{proof}

\begin{remark}
We now have a partial algorithm for computing the rank of $X$.
If $X$ is the zero array then $X$ has rank 0.
If $X$ is a superdiagonal array then $X$ has rank 2.
If $X$ is nonzero and not superdiagonal, and all of its slices are singular, then $X$ has rank 1.
It remains to consider an array $X$ with a non-singular slice; by permuting the directions and the slices,
we may assume that $X_1$ is non-singular.
\end{remark}

\begin{lemma} \label{lemma4} \cite[p.~632-633]{tenBerge}
The rank of a $2 \times 2 \times 2$ array $X$ is at most 3.
\end{lemma}

\begin{proof}
It remains to prove that if the first frontal slice $X_1$ is non-singular, then the rank is at most 3.
We construct an explicit decomposition with $n \le 3$.
We write
  \[
  X_1
  =
  \left[
  \begin{array}{cc}
  x_{11} & x_{12} \\
  x_{21} & x_{22}
  \end{array}
  \right],
  \qquad
  Y_2 = X_2 X_1^{-1}
  =
  \left[
  \begin{array}{cc}
  y_{11} & y_{12} \\
  y_{21} & y_{22}
  \end{array}
  \right].
  \]
Consider the following matrices:
  \begin{alignat*}{2}
  A
  &=
  \left[
  \begin{array}{ccc}
  1 & 0 & y_{12} \\
  0 & 1 & y_{21}
  \end{array}
  \right],
  &\qquad
  B
  &=
  \left[
  \begin{array}{ccc}
  x_{11} & x_{21} & x_{11}+x_{21} \\
  x_{12} & x_{22} & x_{12}+x_{22}
  \end{array}
  \right]
  =
  X_1^t
  \left[
  \begin{array}{ccc}
  1 & 0 & 1 \\
  0 & 1 & 1
  \end{array}
  \right],
  \\
  D
  &=
  \left[
  \begin{array}{ccc}
  1 & 0 & 0 \\
  0 & 1 & 0 \\
  0 & 0 & 0
  \end{array}
  \right],
  &\qquad
  E
  &=
  \left[
  \begin{array}{ccc}
  y_{11}-y_{12} & 0 & 0 \\
  0 & y_{22}-y_{21} & 0 \\
  0 & 0 & 1
  \end{array}
  \right].
  \end{alignat*}
We then verify by direct calculation that
  \begin{align*}
  A D B^t
  &=
  \left[
  \begin{array}{ccc}
  1 & 0 & y_{12} \\
  0 & 1 & y_{21}
  \end{array}
  \right]
  \left[
  \begin{array}{ccc}
  1 & 0 & 0 \\
  0 & 1 & 0 \\
  0 & 0 & 0
  \end{array}
  \right]
  \left[
  \begin{array}{cc}
  1 & 0 \\
  0 & 1 \\
  1 & 1
  \end{array}
  \right]
  X_1
  =
  X_1,
  \\
  A E B^t
  &=
  \left[
  \begin{array}{ccc}
  1 & 0 & y_{12} \\
  0 & 1 & y_{21}
  \end{array}
  \right]
  \left[
  \begin{array}{ccc}
  y_{11}-y_{12} & 0 & 0 \\
  0 & y_{22}-y_{21} & 0 \\
  0 & 0 & 1
  \end{array}
  \right]
  \left[
  \begin{array}{cc}
  1 & 0 \\
  0 & 1 \\
  1 & 1
  \end{array}
  \right]
  X_1
  =
  Y_2 X_1
  =
  X_2.
  \end{align*}
We now apply Lemma \ref{lemma1} to complete the proof.
\end{proof}

\begin{remark}
Lemmas \ref{lemma3} and \ref{lemma4} imply that if $X$ has a non-singular slice then its rank is either 2 or 3.
It remains to distinguish these two cases.
As before, up to permuting the directions and the slices, we may assume that $X_1$ is non-singular.
\end{remark}

\begin{definition}
For a $2 \times 2 \times 2$ array $X$,
\textbf{Cayley's hyperdeterminant}
is the following homogeneous polynomial of degree 4 in the entries $x_{ijk}$:
  \begin{align*}
  \Delta(X)
  &=
  x_{111}^2 x_{222}^2 + x_{112}^2 x_{221}^2 + x_{121}^2 x_{212}^2 + x_{122}^2 x_{211}^2
  \\
  &\quad
  -
  2 \big(
  x_{111} x_{112} x_{221} x_{222}
  + x_{111} x_{121} x_{212} x_{222}
  + x_{111} x_{122} x_{211} x_{222}
  \\
  &\quad\quad\quad
  + x_{112} x_{121} x_{212} x_{221}
  + x_{112} x_{122} x_{211} x_{221}
  + x_{121} x_{122} x_{211} x_{212}
  \big)
  \\
  &\quad
  + 4 \big( x_{111} x_{122} x_{212} x_{221} + x_{112} x_{121} x_{211} x_{222} \big).
  \end{align*}
\end{definition}

\begin{theorem} \label{theorem222} \cite[p.~633-634]{tenBerge}
Let $X$ be a $2 \times 2 \times 2$ array whose first frontal slice $X_1$ is non-singular.
If $X_2$ is a scalar multiple of $X_1$, then $X$ has rank 2.
If $X_2$ is not a scalar multiple of $X_1$, then
(a) if $\Delta(X) \ne 0$ then $X$ has rank 2, and
(b) if $\Delta(X) = 0$ then $X$ has rank 3.
\end{theorem}

\begin{proof}
First, assume that $X_2 = \lambda X_1$ for some $\lambda \in \mathbb{C}$.
Since $X_1$ is non-singular, it has rank 2, and hence
$X_1 = \mathbf{a}^{(1)} \otimes \mathbf{b}^{(1)} + \mathbf{a}^{(2)} \otimes \mathbf{b}^{(2)}$.
Writing $\mathbf{c} = [ \, 1, \lambda \, ]^t$ then we see that $X$ has rank 2:
$X = X_1 \otimes \mathbf{c} =
\mathbf{a}^{(1)} \otimes \mathbf{b}^{(1)} \otimes \mathbf{c}
+
\mathbf{a}^{(2)} \otimes \mathbf{b}^{(2)} \otimes \mathbf{c}$.
Second, assume that $X_2$ is not a scalar multiple of $X_1$. We will find a necessary condition for $X$ to
have rank 2. We apply Lemma \ref{lemma1} with $n = 2$ and write
  \[
  A = \left[ \begin{array}{cc} \mathbf{a}^{(1)} & \mathbf{a}^{(2)} \end{array} \right],
  \quad
  B = \left[ \begin{array}{cc} \mathbf{b}^{(1)} & \mathbf{b}^{(2)} \end{array} \right],
  \quad
  D = \left[ \begin{array}{cc} d_1 & 0 \\ 0 & d_2 \end{array} \right],
  \quad
  E = \left[ \begin{array}{cc} e_1 & 0 \\ 0 & e_2 \end{array} \right].
  \]
But $X_1 = A D B^t$, $X_2 = A E B^t$ gives
$X_1 = d_1 \mathbf{a}_1 \mathbf{b}_1^t + d_2 \mathbf{a}_2 \mathbf{b}_2^t$,
$X_2 = e_1 \mathbf{a}_1 \mathbf{b}_1^t + e_2 \mathbf{a}_2 \mathbf{b}_2^t$.
Since $X_1$ is non-singular, it has rank 2, and so $d_1 \ne 0$, $d_2 \ne 0$.
Since $X_2$ is not a scalar multiple of $X_1$, it follows that $E$ is not a scalar multiple of $D$.
Hence $d_1 e_2 - d_2 e_1 \ne 0$,
and so
$X_2 - d_1^{-1} e_1 X_1$, $X_2 - d_2^{-1} e_2 X_1$ are distinct.
We calculate
  \begin{align*}
  &
  X_2 - d_1^{-1} e_1 X_1
  =
  e_1 \mathbf{a}_1 \mathbf{b}_1^t + e_2 \mathbf{a}_2 \mathbf{b}_2^t
  -
  d_1^{-1} e_1
  \big(
  d_1 \mathbf{a}_1 \mathbf{b}_1^t + d_2 \mathbf{a}_2 \mathbf{b}_2^t
  \big)
  \\
  &=
  e_1 \mathbf{a}_1 \mathbf{b}_1^t + e_2 \mathbf{a}_2 \mathbf{b}_2^t
  -
  e_1 \mathbf{a}_1 \mathbf{b}_1^t - d_1^{-1} d_2 e_1 \mathbf{a}_2 \mathbf{b}_2^t
  =
  d_1^{-1} ( d_1 e_2 - d_2 e_1 ) \mathbf{a}_2 \mathbf{b}_2^t,
  \\
  &
  X_2 - d_2^{-1} e_2 X_1
  =
  e_1 \mathbf{a}_1 \mathbf{b}_1^t + e_2 \mathbf{a}_2 \mathbf{b}_2^t
  -
  d_2^{-1} e_2
  \big(
  d_1 \mathbf{a}_1 \mathbf{b}_1^t + d_2 \mathbf{a}_2 \mathbf{b}_2^t
  \big)
  \\
  &=
  e_1 \mathbf{a}_1 \mathbf{b}_1^t + e_2 \mathbf{a}_2 \mathbf{b}_2^t
  -
  d_1 d_2^{-1} e_2 \mathbf{a}_1 \mathbf{b}_1^t - e_2 \mathbf{a}_2 \mathbf{b}_2^t
  =
  - d_2^{-1} ( d_1 e_2 - d_2 e_1 ) \mathbf{a}_1 \mathbf{b}_1^t.
  \end{align*}
It follows that these two matrices are singular.
Hence the quadratic polynomial $\det( X_2 - \lambda X_1 )$ has two distinct roots,
but this determinant is
  \allowdisplaybreaks
  \begin{align*}
  &
  ( x_{111} x_{221} - x_{121} x_{211} ) \lambda^2
  -
  ( x_{111} x_{222} + x_{112} x_{221} - x_{121} x_{212} - x_{122} x_{211} ) \lambda
  \\
  &\quad
  +
  ( x_{112} x_{222} - x_{122} x_{212} ),
  \end{align*}
and the discriminant is $\Delta(X)$.
Thus if $X$ has rank 2 then $\Delta(X) \ne 0$.

Conversely, suppose that $\Delta(X) \ne 0$.
Then $\det( X_2 - \lambda X_1 )$ has two distinct roots, say $\lambda_1$, $\lambda_2$.
We have two nonzero singular matrices $X_2 - \lambda_1 X_1$, $X_2 - \lambda_2 X_1$.
These matrices both have rank 1, and so we can write
  \[
  ( \lambda_1 - \lambda_2 )^{-1}
  ( X_2 - \lambda_2 X_1 )
  =
  \mathbf{u}_1 \mathbf{v}_1^t,
  \qquad
  -( \lambda_1 - \lambda_2 )^{-1}
  ( X_2 - \lambda_1 X_1 )
  =
  \mathbf{u}_2 \mathbf{v}_2^t.
  \]
Then we have
$X_1 = \mathbf{u}_1 \mathbf{v}_1^t + \mathbf{u}_2 \mathbf{v}_2^t$ and
$X_2 = \lambda_1 \mathbf{u}_1 \mathbf{v}_1^t + \lambda_2 \mathbf{u}_2 \mathbf{v}_2^t$,
which imply that
  $X =
  \mathbf{u}_1 \otimes \mathbf{v}_1 \otimes [ 1, \lambda_1 ]^t
  +
  \mathbf{u}_2 \otimes \mathbf{v}_2 \otimes [ 1, \lambda_2 ]^t.
  $
Thus if $\Delta(X) \ne 0$ then $X$ has rank 2.
\end{proof}

\begin{example} \label{example222rank3}
Consider these arrays, where $X$ is the limit as $a \to 0$ of $Y(a)$:
  \[
  X = \left[
  \begin{array}{cc|cc}
  1 & 0 & 0 & 1 \\
  0 & 1 & 0 & 0
  \end{array}
  \right],
  \qquad
  Y\!(a) =
  \left[
  \begin{array}{cc|cc}
  1 & 0 & 0 & 1 \\
  0 & 1 & a^2 & 0
  \end{array}
  \right].
  \]
Clearly $X_1$ is non-singular, and $X_2$ is not a scalar multiple of $X_1$.
But $\Delta(X) = 0$, and so by Theorem \ref{theorem222} the rank of $X$ is 3.
For $Y(a)$, the first frontal slice is non-singular and the second frontal slice is not a scalar multiple of the first,
but $\Delta(Y\!(a)) = 4a^2$ which is nonzero for $a \ne 0$.
Hence if $a \ne 0$ then $Y\!(a)$ has rank 2.
Thus $X$ is the limit of arrays of rank 2, and so the border rank of $X$ is 2.

It follows from Ja'ja' \cite[Lemma 3.1]{JaJa} that an array $[ \, I \, | \, X_2 ]$ has rank 2 if and only if
$X_2$ is similar to a diagonal matrix.
The same paper \cite[Theorem 3.2]{JaJa} implies that if $X_2$ is the companion matrix
of a quadratic polynomial $f(t)$ then $[ \, I \, | \, X_2 ]$ has rank 2 if and only if $f(t)$ has two distinct roots;
otherwise, it has rank 3.
In our example, $X_2$ is the companion matrix of $f(t) = t^2$, so $[ \, I \, | \, X_2 ]$ has rank 3.
This example is the case $n = 2$ of the pair of bilinear forms in the proof of \cite[Theorem 3.5]{JaJa}.
A result of von zur Gathen \cite[Theorem 4]{vonzurGathen} implies that the maximal bilinear complexity of two $2 \times 2$ matrices
over any field is at least 3.
\end{example}


\section{Some lemmas on $3 \times 3 \times 2$ and $3 \times 3 \times 3$ arrays} \label{lemmasection}

Let the $3 \times 3 \times 2$ array over $\mathbb{C}$ have frontal slices $A$ and $B$:
  \[
  [ A | B ]
  =
  \left[
  \begin{array}{ccc|ccc}
  a_{11} & a_{12} & a_{13} & b_{11} & b_{12} & b_{13} \\
  a_{21} & a_{22} & a_{23} & b_{21} & b_{22} & b_{23} \\
  a_{31} & a_{32} & a_{33} & b_{31} & b_{32} & b_{33}
  \end{array}
  \right]
  \]
Ja'Ja' \cite[Corollary 3.4.1]{JaJa} has shown that the rank of a $p \times p \times 2$ array is at most
$\lfloor 3p/2 \rfloor$.
We give an elementary proof of this result in the case $p = 3$.

\begin{lemma} \label{lemma332}
The rank of a $3 \times 3 \times 2$ array is at most 4.
\end{lemma}

\begin{proof}
The maximum rank of a $3 \times 3$ matrix is 3.
If both $A$ and $B$ have rank $\le 2$, then it is straightforward to express $[A|B]$ as a sum of $\le 4$ simple tensors.
We have
  \[
  A = \mathbf{a}^{(1)} \otimes \mathbf{b}^{(1)} + \mathbf{a}^{(2)} \otimes \mathbf{b}^{(2)},
  \qquad
  B = \mathbf{a}^{(3)} \otimes \mathbf{b}^{(3)} + \mathbf{a}^{(4)} \otimes \mathbf{b}^{(4)},
  \]
and hence
  \[
  [A|B]
  =
  \mathbf{a}^{(1)} \otimes \mathbf{b}^{(1)} \otimes \mathbf{e}_1 + \mathbf{a}^{(2)} \otimes \mathbf{b}^{(2)} \otimes \mathbf{e}_1
  +
  \mathbf{a}^{(3)} \otimes \mathbf{b}^{(3)} \otimes \mathbf{e}_2 + \mathbf{a}^{(4)} \otimes \mathbf{b}^{(4)} \otimes \mathbf{e}_2.
  \]
We now assume that both $A$ and $B$ have rank $\ge 2$, and that either $A$ or $B$ has rank 3.
Interchanging $A$ and $B$ if necessary, we assume that $A$ has rank 3, so that $A$ is invertible.
Left multiplication of $A$ and $B$ by $A^{-1}$ (that is, applying a change of basis in the first direction)
gives the array $[I|C]$ where the second frontal slice $C = A^{-1}B$ still has rank $\ge 2$.
There exists an invertible matrix $E$ such that $J = E^{-1}CE$ is the Jordan canonical form of $C$.
(Here we use the assumption that the base field is algebraically closed.)
Clearly $E^{-1}IE = I$, so we act on $[I|C]$ by $E^{-1}$ along the first direction and by $E$ along the second direction,
to obtain $[I|J]$, where the second frontal slice $J$ still has rank $\ge 2$.
It remains to show that any such array has rank $\le 4$.
There are three cases for the Jordan canonical form of a $3 \times 3$ matrix $J$.
Case 1: Three $1 \times 1$ Jordan blocks; $J$ is a diagonal matrix:
  \[
  [I|J]
  =
  \left[
  \begin{array}{ccc|ccc}
  1 & 0 & 0 & d_1 &  0  &  0 \\
  0 & 1 & 0 &  0  & d_2 &  0 \\
  0 & 0 & 1 &  0  &  0  & d_3
  \end{array}
  \right]
  \]
Then clearly the rank is $\le 3$:
  \[
  [I|J]
  =
  \mathbf{e}_1 \otimes \mathbf{e}_1 \otimes \left[ \begin{array}{c} 1 \\ d_1 \end{array} \right]
  +
  \mathbf{e}_2 \otimes \mathbf{e}_2 \otimes \left[ \begin{array}{c} 1 \\ d_2 \end{array} \right]
  +
  \mathbf{e}_3 \otimes \mathbf{e}_3 \otimes \left[ \begin{array}{c} 1 \\ d_3 \end{array} \right].
  \]
Case 2:
One $2 \times 2$ block and one $1 \times 1$ block:
  \[
  [I|J]
  =
  \left[
  \begin{array}{ccc|ccc}
  1 & 0 & 0 & d_1 &  1  &  0 \\
  0 & 1 & 0 &  0  & d_1 &  0 \\
  0 & 0 & 1 &  0  &  0  & d_2
  \end{array}
  \right]
  \]
We have $[I|J] = [I|D] + [O|F]$ where $D$ is a diagonal matrix and $F$ is the matrix unit $E_{12}$.
By the previous case, $[I|D]$ has rank $\le 3$, and clearly $[O|F]$ has rank 1.
Case 3:
One $3 \times 3$ block:
  \[
  [I|J]
  =
  \left[
  \begin{array}{ccc|ccc}
  1 & 0 & 0 & d_1 &  1  &  0  \\
  0 & 1 & 0 &  0  & d_1 &  1  \\
  0 & 0 & 1 &  0  &  0  & d_1
  \end{array}
  \right]
  \]
We add $-d_1$ times the first frontal slice to the second frontal slice;
that is, we change basis along the third direction by the matrix
  \[
  \begin{bmatrix}
  1 & 0 \\
  -d_1 & 1
  \end{bmatrix}
  \]
We obtain this array:
  \[
  \left[
  \begin{array}{ccc|ccc}
  1 & 0 & 0 & 0 & 1 & 0  \\
  0 & 1 & 0 & 0 & 0 & 1  \\
  0 & 0 & 1 & 0 & 0 & 0
  \end{array}
  \right]
  \]
It remains to prove that this array has rank $\le 4$.
We have the following explicit representation as a sum of four simple tensors:
  \allowdisplaybreaks
  \begin{align*}
  &
  \left[
  \begin{array}{rrr|rrr}
  1 & 0 & 0 & 0 & 1 & 0  \\
  0 & 1 & 0 & 0 & 0 & 1  \\
  0 & 0 & 1 & 0 & 0 & 0
  \end{array}
  \right]
  \\
  &=
  \left[
  \begin{array}{rrr|rrr}
  0 & 1 & 0 & 0 & 1 & 0 \\
  0 & \frac12 & 0 & 0 & \frac12 & 0 \\
  0 & 0 & 0 & 0 & 0 & 0
  \end{array}
  \right]
  +
  \left[
  \begin{array}{rrr|rrr}
  0 & 0 &  0 & 0 &  0 & 0 \\
  0 & \frac12 & -1 & 0 & -\frac12 & 1 \\
  0 & 0 &  0 & 0 &  0 & 0
  \end{array}
  \right]
  \\
  &\qquad
  +
  \left[
  \begin{array}{rrr|rrr}
  1 & -1 & 0 & 0 & 0 & 0 \\
  0 &  0 & 0 & 0 & 0 & 0 \\
  0 &  0 & 0 & 0 & 0 & 0
  \end{array}
  \right]
  +
  \left[
  \begin{array}{rrr|rrr}
  0 & 0 & 0 & 0 & 0 & 0 \\
  0 & 0 & 1 & 0 & 0 & 0 \\
  0 & 0 & 1 & 0 & 0 & 0
  \end{array}
  \right]
  \\
  &=
  \left[ \begin{array}{r} 1 \\ \frac12 \\ 0 \end{array} \right] \otimes
  \left[ \begin{array}{r} 0 \\ 1 \\ 0 \end{array} \right] \otimes
  \left[ \begin{array}{r} 1 \\ 1 \end{array} \right]
  +
  \left[ \begin{array}{r} 0 \\  1 \\ 0 \end{array} \right] \otimes
  \left[ \begin{array}{r} 0 \\ -\frac12 \\ 1 \end{array} \right] \otimes
  \left[ \begin{array}{r}-1 \\  1 \end{array} \right]
  \\
  &\qquad
  +
  \left[ \begin{array}{r} 1 \\  0 \\ 0 \end{array} \right] \otimes
  \left[ \begin{array}{r} 1 \\ -1 \\ 0 \end{array} \right] \otimes
  \left[ \begin{array}{r} 1 \\  0 \end{array} \right]
  +
  \left[ \begin{array}{r} 0 \\ 1 \\ 1 \end{array} \right] \otimes
  \left[ \begin{array}{r} 0 \\ 0 \\ 1 \end{array} \right] \otimes
  \left[ \begin{array}{r} 1 \\ 0 \end{array} \right].
  \end{align*}
This completes the proof.
\end{proof}

Let the $3 \times 3 \times 3$ array $T$ over $\mathbb{C}$ have frontal slices $A$, $B$ and $C$:
  \[
  T
  =
  [ A | B | C ]
  =
  \left[
  \begin{array}{ccc|ccc|ccc}
  a_{11} & a_{12} & a_{13} & b_{11} & b_{12} & b_{13} & c_{11} & c_{12} & c_{13} \\
  a_{21} & a_{22} & a_{23} & b_{21} & b_{22} & b_{23} & c_{21} & c_{22} & c_{23} \\
  a_{31} & a_{32} & a_{33} & b_{31} & b_{32} & b_{33} & c_{31} & c_{32} & c_{33}
  \end{array}
  \right]
  \]

\begin{lemma} \label{lemma1edge}
\textbf{Kruskal's One-Edge Lemma.}
If the array $T$ has parallel slices $D$ and $E$ for which there exists
a nonzero vector $\mathbf{x}$ such that $D \mathbf{x} = E \mathbf{x} = 0$ or $D^t \mathbf{x} = E^t \mathbf{x} = 0$,
then $\mathrm{rank}(T) \le 5$.
\end{lemma}

\begin{proof}
Permuting the directions if necessary, we may assume that $D$ and $E$ are frontal slices.
Permuting the frontal slices if necessary, we may assume that $D$ and $E$ are the first and second frontal slices $A$ and $B$.
Suppose that $A \mathbf{x} = B \mathbf{x} = 0$ where $ \mathbf{x} \ne 0$.
Let $X$ be a $3 \times 3$ non-singular matrix which has $\mathbf{x}$ as its first column.
(Extend the set $\{ \mathbf{x} \}$ to a basis $\{ \mathbf{x}, \mathbf{y}, \mathbf{z} \}$ of $\mathbb{C}^3$ and
let $X = [ \mathbf{x} | \mathbf{y} | \mathbf{z} ]$.)
Acting on $T = [A|B|C]$ by $X$ along the second direction gives $[ AX | BX | CX ]$, but $A \mathbf{x} = B \mathbf{x} = 0$, so
  \allowdisplaybreaks
  \begin{align*}
  &
  [ AX | BX | CX ]
  =
  \left[
  \begin{array}{ccc|ccc|ccc}
  0 &\! a'_{12} &\! a'_{13} &\! 0 &\! b'_{12} &\! b'_{13} &\! c'_{11} &\! c'_{12} &\! c'_{13} \\
  0 &\! a'_{22} &\! a'_{23} &\! 0 &\! b'_{22} &\! b'_{23} &\! c'_{21} &\! c'_{22} &\! c'_{23} \\
  0 &\! a'_{32} &\! a'_{33} &\! 0 &\! b'_{32} &\! b'_{33} &\! c'_{31} &\! c'_{32} &\! c'_{33}
  \end{array}
  \right]
  \\
  &=
  \left[
  \begin{array}{ccc|ccc|ccc}
  0 &\! 0 &\! 0 &\! 0 &\! 0 &\! 0 &\! c'_{11} &\! 0 &\! 0 \\
  0 &\! 0 &\! 0 &\! 0 &\! 0 &\! 0 &\! c'_{21} &\! 0 &\! 0 \\
  0 &\! 0 &\! 0 &\! 0 &\! 0 &\! 0 &\! c'_{31} &\! 0 &\! 0
  \end{array}
  \right]
  +
  \left[
  \begin{array}{ccc|ccc|ccc}
  0 &\! a'_{12} &\! a'_{13} &\! 0 &\! b'_{12} &\! b'_{13} &\! 0 &\! c'_{12} &\! c'_{13} \\
  0 &\! a'_{22} &\! a'_{23} &\! 0 &\! b'_{22} &\! b'_{23} &\! 0 &\! c'_{22} &\! c'_{23} \\
  0 &\! a'_{32} &\! a'_{33} &\! 0 &\! b'_{32} &\! b'_{33} &\! 0 &\! c'_{32} &\! c'_{33}
  \end{array}
  \right]
  \end{align*}
The first term is a simple tensor,
  \[
  \left[ \begin{array}{c} c'_{11} \\ c'_{21} \\ c'_{31} \end{array} \right]
  \otimes
  \left[ \begin{array}{c} 1 \\ 0 \\ 0 \end{array} \right]
  \otimes
  \left[ \begin{array}{c} 0 \\ 0 \\ 1 \end{array} \right],
  \]
and so it remains to prove that the second term has rank $\le 4$.
To write the second term as a sum of simple tensors it suffices to decompose
this $3 \times 2 \times 3$ array:
  \[
  \left[
  \begin{array}{cc|cc|cc}
  a'_{12} & a'_{13} & b'_{12} & b'_{13} & c'_{12} & c'_{13} \\
  a'_{22} & a'_{23} & b'_{22} & b'_{23} & c'_{22} & c'_{23} \\
  a'_{32} & a'_{33} & b'_{32} & b'_{33} & c'_{32} & c'_{33}
  \end{array}
  \right]
  \]
Transposing the second and third directions, we may consider the $3 \times 3 \times 2$ array,
  \[
  \left[
  \begin{array}{ccc|ccc}
  a'_{12} & b'_{12} & c'_{12} & a'_{13} & b'_{13} & c'_{13} \\
  a'_{22} & b'_{22} & c'_{22} & a'_{23} & b'_{23} & c'_{23} \\
  a'_{32} & b'_{32} & c'_{32} & a'_{33} & b'_{33} & c'_{33}
  \end{array}
  \right]
  \]
The claim now follows from Lemma \ref{lemma332}.

If $A^t \mathbf{x} = B^t \mathbf{x} = 0$, then we transpose the matrices $A$, $B$ and $C$ and use the analogous reasoning;
this can also be expressed in terms of a transposition of the first two directions in the array $T = [A|B|C]$.
\end{proof}

\begin{lemma} \label{lemma2edge}
\textbf{Kruskal's Two-Edge Lemma.}
If the array $T$ has frontal slices $D$ and $E$ for which there exist
nonzero vectors $\mathbf{x}$ and $\mathbf{y}$ such that $D \mathbf{x} = \mathbf{y}^t D = 0$
and $\mathbf{y}^t E \mathbf{x} \ne 0$, then $\mathrm{rank}(T) \le 5$.
\end{lemma}

\begin{proof}
As before, we may assume that $D$ and $E$ are the first and second frontal slices $A$ and $B$.
Let $\mathbf{x}$ and $\mathbf{y}$ satisfy the conditions of the lemma.
We choose vectors $\mathbf{u}_2$, $\mathbf{u}_3$, $\mathbf{v}_2$, $\mathbf{v}_3$ such that
$U = [ \, \mathbf{x} | \mathbf{u}_2 | \mathbf{u}_3 \, ]$ and $V = [ \, \mathbf{y} | \mathbf{v}_2 | \mathbf{v}_3 \, ]$ are nonsingular.
Then
  \[
  V^t A U = \begin{bmatrix} 0 & 0 & 0 \\ 0 & \ast & \ast \\ 0 & \ast & \ast \end{bmatrix},
  \quad
  V^t B U = \begin{bmatrix} \alpha & \ast & \ast \\ \ast & \ast & \ast \\ \ast & \ast & \ast \end{bmatrix},
  \quad \text{and} \quad
  V^t C U = \begin{bmatrix} \beta & \ast & \ast \\ \ast & \ast & \ast \\ \ast & \ast & \ast \end{bmatrix},
  \]
where $\alpha = \mathbf{y}^t B \mathbf{x} \ne 0$, but $\beta$ can be 0, and $\ast$ denotes unspecified elements (which are not necessarily equal).
If $\beta = 0$ then we add $B$ to $C$ to make $\beta = \alpha \ne 0$.
Equivalently, we change basis along the third direction in $T$ by the matrix
  \[
  \begin{bmatrix}
  1 & 0 & 0 \\
  0 & 1 & 0 \\
  0 & 1 & 1
  \end{bmatrix}
  \]
Since $\alpha \ne 0$, we can construct a matrix $X$ of rank 1 which has the same first row and first column as $V^t B U$;
explicitly,
  \[
  V^t B U = \begin{bmatrix} \alpha & \alpha' & \alpha'' \\ \gamma & \ast & \ast \\ \delta & \ast & \ast \end{bmatrix},
  \qquad
  X = \begin{bmatrix}
  \alpha & \alpha' & \alpha'' \\
  \frac{\gamma}{\alpha}\alpha & \frac{\gamma}{\alpha}\alpha' & \frac{\gamma}{\alpha}\alpha'' \\[4pt]
  \frac{\delta}{\alpha}\alpha & \frac{\delta}{\alpha}\alpha' & \frac{\delta}{\alpha}\alpha''
  \end{bmatrix}
  \]
Similarly, we can construct a matrix $Y$ of rank 1 which has the same first row and first column as $V^t C U$.
Then the two arrays $[0|X|0]$ and $[0|Y|0]$ also have rank 1 as $3 \times 3 \times 3$ arrays;
that is, they are simple tensors.
We now see that
  \[
  V^t [A|B|C] U - [0|X|0] - [0|Y|0]
  =
  \left[
  \begin{array}{ccc|ccc|ccc}
  0 & 0 & 0 & 0 & 0 & 0 & 0 & 0 & 0 \\
  0 & \ast & \ast & 0 & \ast & \ast & 0 & \ast & \ast \\
  0 & \ast & \ast & 0 & \ast & \ast & 0 & \ast & \ast
  \end{array}
  \right]
  \]
It remains to decompose a $2 \times 2 \times 2$ array, and this requires at most three simple tensors
according to the results of Section \ref{tenBergesection}.
\end{proof}


\section{Proof of Kruskal's theorem on $3 \times 3 \times 3$ arrays} \label{Kruskalsection}

\begin{theorem}
Every $3 \times 3 \times 3$ array $T = [A|B|C]$ over $\mathbb{C}$ has rank $\le 5$.
\end{theorem}

\begin{proof}
If $C = 0$ then the problem reduces to considering a $3 \times 3 \times 2$ array, which has rank $\le 4$
by Lemma \ref{lemma332}.
We assume from now on that $C \ne 0$.
Consider the $3 \times 3$ matrix $A - \lambda C$; its determinant is a nonconstant polynomial in $\lambda$,
which has a root over $\mathbb{C}$.
(Here again we use the assumption that the base field is algebraically closed.)
Thus by subtracting a multiple of $C$ from $A$, we may ensure that $A$ is singular, and so
$\mathrm{rank}(A) \le 2$.
Equivalently, we change basis along the third direction in $T$ by the matrix
  \[
  \begin{bmatrix}
  1 & 0 & -\lambda \\
  0 & 1 & 0 \\
  0 & 0 & 1
  \end{bmatrix}
  \]
The same considerations apply to $B$.
We assume from now on that the first and second frontal slices of $T$ both have rank $\le 2$.

Suppose that some frontal slice has rank $\le 1$; up to permuting these slices, we may assume
that $\mathrm{rank}(A) \le 1$.
If $\mathrm{rank}(A) = 0$, then $A$ is the 0 matrix, and we have a $3 \times 3 \times 2$ array, which has rank $\le 4$
by Lemma \ref{lemma332}.
If $\mathrm{rank}(A) = 1$, then the array $[A|0|0]$ has rank 1;
subtracting this simple tensor from $T$ leaves a $3 \times 3 \times 2$ array which has rank $\le 4$,
and so $T$ has rank $\le 5$.

We may now assume that $A$ and $B$ have rank 2, and that $C$ has rank 2 or 3.

\smallskip

\textsc{Case 1:}
$A$, $B$ and $C$ all have rank 2.
It follows that there exist nonzero vectors $\mathbf{x}_1, \mathbf{x}_2, \mathbf{x}_3$ and $\mathbf{y}_1, \mathbf{y}_2, \mathbf{y}_3$
(basis vectors for the right and left nullspaces) such that
  \[
  A \mathbf{x}_1 = B \mathbf{x}_2 = C \mathbf{x}_3 = 0,
  \qquad
  \mathbf{y}_1^t A = \mathbf{y}_2^t B = \mathbf{y}_3^t C = 0.
  \]
Then for the $3 \times 3$ matrices $X = [ \mathbf{x}_1 | \mathbf{x}_2 | \mathbf{x}_3 ]$ and $Y = [ \mathbf{y}_1 | \mathbf{y}_2 | \mathbf{y}_3 ]$ we have
  \begin{equation} \label{case1eq1}
  Y^t A X = \begin{bmatrix} 0 & 0 & 0 \\ 0 & \ast & \ast \\ 0 & \ast & \ast \end{bmatrix},
  \quad
  Y^t B X = \begin{bmatrix} \ast & 0 & \ast \\ 0 & 0 & 0 \\ \ast & 0 & \ast \end{bmatrix},
  \quad
  Y^t C X = \begin{bmatrix} \ast & \ast & 0 \\ \ast & \ast & 0 \\ 0 & 0 & 0 \end{bmatrix}.
  \end{equation}
If the conditions of Lemma \ref{lemma2edge} are satisfied for any two frontal slices, then the proof is complete.
Otherwise, it follows that
  \begin{equation} \label{case1eq2}
  \mathbf{y}_i^t A \mathbf{x}_i = \mathbf{y}_i^t B \mathbf{x}_i = \mathbf{y}_i^t C \mathbf{x}_i = 0,
  \;
  \text{for all $i = 1, 2, 3$}.
  \end{equation}
Consider these three subcases:

\emph{Subcase 1.1:}
Two columns of $X$ are linearly dependent (that is, one column is a scalar multiple of another).
Then Lemma \ref{lemma1edge} completes the proof.

\emph{Subcase 1.2:}
The matrix $X$ has rank 2, but no two columns are linearly dependent.
Then $\mathbf{x}_1$ and $\mathbf{x}_2$ are linearly independent, and so $\mathbf{x}_3 = \beta \mathbf{x}_1 + \gamma \mathbf{x}_2$
for some $\beta, \gamma \in \mathbb{C} \setminus \{0\}$.
We choose vectors $\mathbf{u}, \mathbf{v}_2, \mathbf{v}_3$ such that the matrices
$U = [ \mathbf{x}_1 | \mathbf{x}_2 | \mathbf{u} ]$ and $V = [ \mathbf{y}_1 | \mathbf{v}_2 | \mathbf{v}_3 ]$
are invertible.
Then for some $\delta \in \mathbb{C}$ we have
  \[
  V^t A U = \begin{bmatrix} 0 & 0 & 0 \\ 0 & \ast & \ast \\ 0 & \ast & \ast \end{bmatrix},
  \qquad
  V^t B U = \begin{bmatrix} 0 & 0 & \ast \\ \ast & 0 & \ast \\ \ast & 0 & \ast \end{bmatrix},
  \qquad
  V^t C U = \begin{bmatrix} 0 & \delta & \ast \\ \ast & \ast & \ast \\ \ast & \ast & \ast \end{bmatrix}.
  \]
(The $(1,1)$ entries of $V^t B U$ and $V^t C U$ are zero; otherwise Lemma \ref{lemma2edge} would apply.)
But $C \mathbf{x}_3 = 0$ implies $\beta C \mathbf{x}_1 + \gamma C \mathbf{x}_2 = 0$, and so the first two columns of $V^t C U$ are linearly dependent,
implying $\delta = 0$.
Hence the first three rows of $V^t A U$, $V^t B U$ and $V^t C U$ are linearly dependent.
We subtract the simple tensor in which the first horizontal slice is the same as that of $V^t [A|B|C] U$ and
the second and third horizontal slices are zero.
There remains an array in which the first horizontal slice is zero, and the second and third horizontal slices are
the same as those of $V^t [A|B|C] U$.
But the rank of this $2 \times 3 \times 3$ array is at most 4 by Lemma \ref{lemma332}.

\emph{Subcase 1.3:}
The matrix $X$ has rank 3.
If $Y$ has rank $\le 2$, then we replace each frontal slice $A$, $B$, $C$ by its transpose
(equivalently, we interchange the first two directions of $T$), and then we may apply one of the previous subcases.
So we assume that $Y$ has rank 3.
Combining \eqref{case1eq1} and \eqref{case1eq2} gives three matrices of rank 2:
  \[
  Y^t A X = \begin{bmatrix} 0 & 0 & 0 \\ 0 & 0 & \beta \\ 0 & \alpha & 0 \end{bmatrix},
  \quad
  Y^t B X = \begin{bmatrix} 0 & 0 & \delta \\ 0 & 0 & 0 \\ \gamma & 0 & 0 \end{bmatrix},
  \quad
  Y^t C X = \begin{bmatrix} 0 & \zeta & 0 \\ \epsilon & 0 & 0 \\ 0 & 0 & 0 \end{bmatrix}.
  \]
If we transpose the first two columns of $X$ (that is, interchange the first two vertical slices of $T$) then we obtain
  \[
  Y^t A X = \begin{bmatrix} 0 & 0 & 0 \\ 0 & 0 & \beta \\ \alpha & 0 & 0 \end{bmatrix},
  \quad
  Y^t B X = \begin{bmatrix} 0 & 0 & \delta \\ 0 & 0 & 0 \\ 0 & \gamma & 0 \end{bmatrix},
  \quad
  Y^t C X = \begin{bmatrix} \zeta & 0 & 0 \\ 0 & \epsilon & 0 \\ 0 & 0 & 0 \end{bmatrix}.
  \]
We subtract from $Y^t [A|B|C] X$ the following sum of three simple tensors:
  \allowdisplaybreaks
  \begin{align*}
  &
  \left[
  \begin{array}{rrr|rrr|rrr}
  0 &\!\!\! 0 &\!\!\! 0 & 0 & 0 & 0 & 0 & 0 & 0 \\
  0 &\!\!\! 0 &\!\!\! 0 & 0 & 0 & 0 & 0 & 0 & 0 \\
  \alpha &\!\!\! -\beta &\!\!\! -\beta & 0 & 0 & 0 & 0 & 0 & 0
  \end{array}
  \right]
  +
  \\
  &
  \left[
  \begin{array}{rrr|rrr|rrr}
  0 & 0 & 0 & 0 & 0 &\!\!\! \delta & 0 & 0 & 0 \\
  0 & 0 & 0 & 0 & 0 &\!\!\! -\gamma & 0 & 0 & 0 \\
  0 & 0 & 0 & 0 & 0 &\!\!\! -\gamma & 0 & 0 & 0
  \end{array}
  \right]
  +
  \left[
  \begin{array}{rrr|rrr|rrr}
  0 & 0 & 0 & 0 & 0 & 0 & \zeta & 0 & 0 \\
  0 & 0 & 0 & 0 & 0 & 0 & 0 & 0 & 0 \\
  0 & 0 & 0 & 0 & 0 & 0 & 0 & 0 & 0
  \end{array}
  \right]
  \end{align*}
We obtain the following array:
  \begin{equation}
  \label{result1}
  \left[
  \begin{array}{rrr|rrr|rrr}
  0 &     0 &     0 & 0 &      0 &      0 & 0 & 0        & 0 \\
  0 &     0 & \beta & 0 &      0 & \gamma & 0 & \epsilon & 0 \\
  0 & \beta & \beta & 0 & \gamma & \gamma & 0 & 0        & 0
  \end{array}
  \right]
  \end{equation}
The three frontal slices are linear combinations of these two matrices of rank 1:
  \[
  \begin{bmatrix} 0 & 0 & 0 \\ 0 & 1 & 1 \\ 0 & 1 & 1 \end{bmatrix},
  \qquad
  \begin{bmatrix} 0 & 0 & 0 \\ 0 & 1 & 0 \\ 0 & 0 & 0 \end{bmatrix}.
  \]
Therefore \eqref{result1} is the sum of two simple tensors:
  \[
  \left[
  \begin{array}{rrr|rrr|rrr}
  0 & 0     & 0     & 0 & 0      & 0      & 0 & 0 & 0 \\
  0 & \beta & \beta & 0 & \gamma & \gamma & 0 & 0 & 0 \\
  0 & \beta & \beta & 0 & \gamma & \gamma & 0 & 0 & 0
  \end{array}
  \right]
  +
  \left[
  \begin{array}{rrr|rrr|rrr}
  0 &\!\!\! 0      & 0 & 0 &\!\!\! 0       & 0 & 0 & 0 & 0 \\
  0 &\!\!\! -\beta & 0 & 0 &\!\!\! -\gamma & 0 & 0 & \epsilon & 0 \\
  0 &\!\!\! 0      & 0 & 0 &\!\!\! 0       & 0 & 0 & 0 & 0
  \end{array}
  \right]
  \]
We now have a decomposition of the original $3 \times 3 \times 3$ array $T$ into a sum of
at most 5 simple tensors.

\smallskip

\textsc{Case 2:}
$A$ and $B$ have rank 2 but $C$ has rank 3.

\emph{Subcase 2.1:}
There exist $\alpha, \beta \in \mathbb{C}$ such that $\alpha A + \beta B + C$ has rank 2.
This corresponds to changing basis in $T$ along the third direction by the matrix
  \[
  \begin{bmatrix}
  1 & 0 & 0 \\
  0 & 1 & 0 \\
  \alpha & \beta & 1
  \end{bmatrix}
  \]
Then we are back in Case 1.
Such scalars may not exist; a simple example is
  \[
  A = B = \begin{bmatrix} 0 & 1 & 0 \\ 0 & 0 & 1 \\ 0 & 0 & 0 \end{bmatrix},
  \qquad
  C = \begin{bmatrix} 1 & 0 & 0 \\ 0 & 1 & 0 \\ 0 & 0 & 1 \end{bmatrix}.
  \]

\emph{Subcase 2.2:}
The matrix $\alpha A + \beta B + C$ has rank 3 for all $\alpha, \beta \in \mathbb{C}$.
There exist nonzero vectors $\mathbf{x}_1, \mathbf{x}_2, \mathbf{x}_3$ and $\mathbf{y}_1, \mathbf{y}_2, \mathbf{y}_3$ such that
  \[
  A \mathbf{x}_1 = B \mathbf{x}_2 = 0,
  \qquad
  C \mathbf{x}_3 = A \mathbf{x}_2,
  \qquad
  \mathbf{y}_1^t A = \mathbf{y}_2^t B = 0,
  \qquad
  \mathbf{y}_3^t C = \mathbf{y}_2^t A.
  \]
If we can apply Lemma \ref{lemma2edge}, then we are done.
So we may assume that Lemma \ref{lemma2edge} does not apply, and hence we must have
  \[
  \mathbf{y}_1^t A \mathbf{x}_1 = \mathbf{y}_1^t B \mathbf{x}_1 = \mathbf{y}_1^t C \mathbf{x}_1 = 0,
  \qquad
  \mathbf{y}_2^t A \mathbf{x}_2 = \mathbf{y}_2^t B \mathbf{x}_2 = \mathbf{y}_2^t C \mathbf{x}_2 = 0.
  \]
We write $X = [\mathbf{x}_1|\mathbf{x}_2|\mathbf{x}_3]$ and $Y = [\mathbf{y}_1|\mathbf{y}_2|\mathbf{y}_3]$.

\emph{Subcase 2.2.1:}
We have linear dependence of $\mathbf{x}_1$ and $\mathbf{x}_2$, or of $\mathbf{y}_1$ and $\mathbf{y}_2$, or both.
Then the result follows from Lemma \ref{lemma1edge}.

\emph{Subcase 2.2.2:}
We have linear independence of $\mathbf{x}_1$ and $\mathbf{x}_2$, and of $\mathbf{y}_1$ and $\mathbf{y}_2$.
Then both matrices $X$ and $Y$ have rank $\ge 2$.

Assume $X$ has rank 2.
Then $\mathbf{x}_3 = \gamma \mathbf{x}_1 + \delta \mathbf{x}_2$ for some $\gamma, \delta \in \mathbb{C}$.
There exist nonzero vectors $\mathbf{u}$, $\mathbf{v}$ such that
$U = [\mathbf{x}_1|\mathbf{x}_2|\mathbf{u}]$, $V = [\mathbf{y}_1|\mathbf{y}_2|\mathbf{v}]$ both have rank 3.
Then using the previous equations we have
  \[
  V^t A U = \begin{bmatrix} 0 & 0 & 0 \\ 0 & 0 & \ast \\ 0 & \ast & \ast \end{bmatrix},
  \qquad
  V^t B U = \begin{bmatrix} 0 & 0 & \zeta \\ 0 & 0 & 0 \\ \epsilon & 0 & \ast \end{bmatrix},
  \qquad
  V^t C U = \begin{bmatrix} 0 & \theta & \ast \\ \eta & 0 & \ast \\ \ast & \ast & \ast \end{bmatrix}.
  \]
Since $B$ has rank 2, it follows that $\epsilon \ne 0$ and $\zeta \ne 0$.
We have
  \[
  V^t A \mathbf{x}_2 = V^t C \mathbf{x}_3 = V^t C ( \gamma \mathbf{x}_1 + \delta \mathbf{x}_2 ) = \gamma V^t C \mathbf{x}_1 + \delta V^t C \mathbf{x}_2,
  \]
and so the second column of $V^t A U$ is a linear combination of the first two columns of $V^t C U$.
Since $\mathbf{x}_3 \ne 0$, it follows that $\gamma$ and $\delta$ are not both 0, and so $\eta = 0$ or $\theta = 0$ (or both).
In either case, adding a multiple of $B$ to $C$ (that is, changing basis along the third direction),
and applying the same change of basis matrices $V^t$ and $U$ along the first and second directions,
gives an array in which the third frontal slice has rank 2.
This contradicts the assumption that $\alpha A + \beta B + C$ has rank 3.

Assume $X$ has rank 3.
If $Y$ has rank 2, then we interchange the first and second directions of the array,
which amounts to applying the usual matrix transpose to the frontal slices $A, B, C$.
Equivalently, we interchange $X$ and $Y$, which reduces to the previous paragraph.
So we may assume that $Y$ also has rank 3.

Using the previous equations, together with
  \[
  \mathbf{y}_2^t A \mathbf{x}_3 = \mathbf{y}_3^t C \mathbf{x}_3 = \mathbf{y}_3^t A \mathbf{x}_2,
  \]
we obtain
  \[
  Y^t A X = \begin{bmatrix} 0 & 0 & 0 \\ 0 & 0 & \gamma \\ 0 & \gamma & \delta \end{bmatrix}.
  \]
Using the previous equations, we obtain
  \[
  Y^t B X = \begin{bmatrix} 0 & 0 & \zeta \\ 0 & 0 & 0 \\ \epsilon & 0 & \eta \end{bmatrix}.
  \]
Using the previous equations, together with
  \begin{alignat*}{3}
  \mathbf{y}_1^t C \mathbf{x}_1 &= 0,
  &\qquad
  \mathbf{y}_1^t C \mathbf{x}_3 &= \mathbf{y}_1^t A \mathbf{x}_2 = 0,
  \\
  \mathbf{y}_2^t C \mathbf{x}_2 &= 0,
  &\qquad
  \mathbf{y}_2^t C \mathbf{x}_3 &= \mathbf{y}_2^t A \mathbf{x}_2 = 0,
  \\
  \mathbf{y}_3^t C \mathbf{x}_1 &= \mathbf{y}_2^t A \mathbf{x}_1 = 0,
  &\qquad
  \mathbf{y}_3^t C \mathbf{x}_2 &= \mathbf{y}_2^t A \mathbf{x}_2 = 0,
  \end{alignat*}
we obtain
  \[
  Y^t C X = \begin{bmatrix} 0 & \lambda & 0 \\ \kappa & 0 & 0 \\ 0 & 0 & \gamma \end{bmatrix}.
  \]
If $\delta \ne 0$ (respectively $\eta \ne 0$) then we add a multiple of $A$ (respectively $B$) to $C$
to eliminate $\gamma$ and obtain an array for which $Y^t C X$ has rank 2; but this contradicts the assumption that
$\alpha A + \beta B + C$ has rank 3.
So we may assume that $\delta = \eta = 0$:
  \[
  Y^t A X = \begin{bmatrix} 0 & 0 & 0 \\ 0 & 0 & \gamma \\ 0 & \gamma & 0 \end{bmatrix},
  \quad
  Y^t B X = \begin{bmatrix} 0 & 0 & \zeta \\ 0 & 0 & 0 \\ \epsilon & 0 & 0 \end{bmatrix},
  \quad
  Y^t C X = \begin{bmatrix} 0 & \lambda & 0 \\ \kappa & 0 & 0 \\ 0 & 0 & \gamma \end{bmatrix}.
  \]
Interchanging the first and second vertical slices of $T = [A|B|C]$, and applying the same transformations,
amounts to interchanging the first and second columns in each of the above matrices.  We now have this array:
  \[
  \left[
  \begin{array}{ccc|ccc|ccc}
  0      & 0 & 0      & 0 & 0        & \zeta & \lambda & 0      & 0      \\
  0      & 0 & \gamma & 0 & 0        & 0     & 0       & \kappa & 0      \\
  \gamma & 0 & 0      & 0 & \epsilon & 0     & 0       & 0      & \gamma
  \end{array}
  \right].
  \]
But $\lambda$, $\kappa$, $\gamma$ are all nonzero by our assumption that the third frontal slice has rank 3.
We scale the first, second and third horizontal slices by $1/\lambda$, $1/\kappa$, $1/\gamma$ respectively:
  \[
  \left[
  \begin{array}{ccc|ccc|ccc}
  0 & 0 & 0                 & 0 & 0                   & \zeta/\lambda & 1 & 0 & 0 \\
  0 & 0 & \gamma/\kappa & 0 & 0                   & 0                 & 0 & 1 & 0 \\
  1 & 0 & 0                 & 0 & \epsilon/\gamma & 0                 & 0 & 0 & 1
  \end{array}
  \right].
  \]
From this array we subtract the following array of rank 2, where the bars denote complex conjugates:
  \begin{align*}
  &
  \left[
  \begin{array}{ccc|ccc|ccc}
  0 & 0 & 0      & 0 & 0        & 0 & 0 & 0 & 0 \\
  0 & 0 & \gamma/\kappa & 0 & 0        & -\overline{\epsilon/\gamma}     & 0 & 0 & 0 \\
  1 & 0 & 0      & -\overline{\zeta/\lambda} & 0 & 0     & 0 & 0 & 0
  \end{array}
  \right]
  =
  \\
  &
  \left[ \begin{array}{c} 0 \\ 1 \\ 0 \end{array} \right]
  \otimes
  \left[ \begin{array}{c} 0 \\ 0 \\ 1 \end{array} \right]
  \otimes
  \left[ \begin{array}{c} \gamma/\kappa \\ -\overline{\epsilon/\gamma} \\ 0 \end{array} \right]
  +
  \left[ \begin{array}{c} 0 \\ 0 \\ 1 \end{array} \right]
  \otimes
  \left[ \begin{array}{c} 1 \\ 0 \\ 0 \end{array} \right]
  \otimes
  \left[ \begin{array}{c} 1 \\ -\overline{\zeta/\lambda} \\ 0 \end{array} \right],
  \end{align*}
and obtain
  \[
  \left[
  \begin{array}{ccc|ccc|ccc}
  0 & 0 & 0 & 0     & 0        & \zeta/\lambda    & 1 & 0 & 0 \\
  0 & 0 & 0 & 0     & 0        & \overline{\epsilon/\gamma} & 0 & 1 & 0 \\
  0 & 0 & 0 & \overline{\zeta/\lambda} & \epsilon/\gamma & 0        & 0 & 0 & 1
  \end{array}
  \right].
  \]
The second frontal slice (which we still denote by $B$) is now Hermitian, and so its Jordan canonical form $J = E^{-1} B E$ is a diagonal matrix.
Changing basis along the first and second directions by $E^{-1}$ and $E$ respectively, we obtain
  \[
  \left[
  \begin{array}{ccc|ccc|ccc}
  0 & 0 & 0 & \mu & 0   & 0   & 1 & 0 & 0 \\
  0 & 0 & 0 & 0   & \nu & 0   & 0 & 1 & 0 \\
  0 & 0 & 0 & 0   & 0   & \xi & 0 & 0 & 1
  \end{array}
  \right].
  \]
This array clearly has rank 3, and the proof is complete.
\end{proof}


\section{Arrays over the field with two elements} \label{F2section}

In this section we use computer algebra to classify the canonical forms of $3 \times 3 \times 3$ arrays $X = [ x_{ijk} ]$
over the field $\mathbb{F}_2$ with two elements.
We use the term tensor for such an array to avoid confusion with the data structures called arrays in Maple.
The flattening of $X$ is the row vector
$\mathrm{flat}(X) = [ x_{111}, \dots, x_{ijk}, \dots, x_{333} ]$,
where the entries are in lex order by subscripts.
Conversely, the unflattening of such a row vector is the corresponding tensor.
We encode $X$ as the non-negative integer whose representation in base 2 is $\mathrm{flat}(X)$.
Conversely, the decoding of an integer in the range $0, \dots, 2^{27}{-}1$ is the
corresponding tensor.
The lex order on flattenings coincides with the natural order on integers.
The minimal element of a set of tensors is defined in terms of this total order.
We identify $X$ with an element of $\mathbb{F}_2^3 \otimes \mathbb{F}_2^3 \otimes \mathbb{F}_2^3$.
The direct product of general linear groups $GL_3(\mathbb{F}_2) \times GL_3(\mathbb{F}_2) \times GL_3(\mathbb{F}_2)$ acts on
$\mathbb{F}_2^3 \otimes \mathbb{F}_2^3 \otimes \mathbb{F}_2^3$,
and the canonical form of a tensor is the minimal element in its orbit under this group action.
The finite group $GL_3(\mathbb{F}_2)$ has order 168, and is generated by two elements:
the cyclic permutation
$e_1 \mapsto e_2$, $e_2 \mapsto e_3$, $e_3 \mapsto e_1$,
and the row operation
$e_1 \mapsto e_1 + e_2$, $e_2 \mapsto e_2$, $e_3 \mapsto e_3$.
The group $GL_3(\mathbb{F}_2) \times GL_3(\mathbb{F}_2) \times GL_3(\mathbb{F}_2)$
has order 4741632 and is generated by 6 elements.

For a tensor $X$ over $\mathbb{F}_2$, we use the spinning algorithm to compute its orbit.
In the following pseudocode, $\mathcal{O}$ is the current value of the orbit,
$\mathcal{L}$ contains the new elements computed during the previous iteration,
and $\mathcal{N}$ contains the new elements computed during the current iteration:
  \begin{enumerate}
  \item
  $\mathcal{O} \leftarrow \emptyset$;
  $\mathcal{L} \leftarrow \{ X \}$
  \item
  while $\mathcal{L} \ne \emptyset$ do:
    \begin{enumerate}
    \item
    $\mathcal{O} \leftarrow \mathcal{O} \cup \mathcal{L}$
    \item
    $\mathcal{N} \leftarrow \emptyset$;
    for $Y \in \mathcal{L}$ do for $M \in \mathcal{G}$ do:
      $\mathcal{N} \leftarrow \mathcal{N} \cup \{ M \cdot Y \}$
    \item
    $\mathcal{L} \leftarrow \mathcal{N} \setminus \mathcal{O}$
    \end{enumerate}
  \item
  return $\mathcal{O}$
  \end{enumerate}
We first create a large Maple array, called \texttt{orbitarray}, with $2^{27}{-}1$ entries.
The indices of \texttt{orbitarray} correspond to nonzero tensors: for an index $i$
we first decode $i$ by writing it as a binary numeral of 27 bits (adding leading 0s if necessary),
and then unflatten this binary numeral to obtain the corresponding tensor.
To start, every entry of \texttt{orbitarray} is set to 0.
We then perform the following iteration:
  \begin{enumerate}
  \item
  $\omega \leftarrow 0$, $i \leftarrow 0$
  \item
  while $i < 2^{27}{-}1$ do:
    \begin{enumerate}
    \item
    $i \leftarrow i + 1$
    \item
    if $\texttt{orbitarray}[i] = 0$ then
      \begin{enumerate}
      \item
      $\omega \leftarrow \omega + 1$
      \item
      $\texttt{findorbit}[i]$
      \end{enumerate}
    \end{enumerate}
  \end{enumerate}
Procedure \texttt{findorbit} takes the index $i$, decodes and unflattens it
to the corresponding tensor $X$,
uses the spinning algorithm to generate the orbit $\mathcal{O}(X)$, and
sets the corresponding entries of \texttt{orbitarray} to the orbit index $\omega$.
Upon termination, $\omega$ equals the total number of orbits for the group
action, and \texttt{orbitarray} represents the function which assigns to each tensor the index
number of its orbit.
The natural order of the index numbers of the orbits agrees with the lex order
on the minimal elements in the orbits (the canonical forms of the tensors).

The next step is to compute the ranks of the orbits.
We create another Maple array, called \texttt{linkarray}, of the same size as \texttt{orbitarray}.
We use the data from \texttt{orbitarray} to set entry $i$ of \texttt{linkarray}
(representing the tensor $X$)
equal to the index $j$ of the next tensor in lex order in the orbit containing $X$.
We then create another Maple array of the same size, called \texttt{rankarray},
and initialize every entry to 0.
We generate all simple tensors (tensor products of nonzero vectors) and
set the corresponding entries of \texttt{rankarray} to 1.
Each index $i$ for which $\texttt{rankarray}[i] = 1$ represents the encoding of a tensor of rank 1.
Let $E$ denote the minimal tensor of rank 1: its flattening is $[0,\dots,0,1]$.
We then perform the following iteration:
  \begin{enumerate}
  \item
  $\texttt{oldrank} \leftarrow 0$, $\texttt{finished} \leftarrow \textrm{false}$
  \item
  While not \texttt{finished} do:
    \begin{enumerate}
    \item
    $\texttt{oldrank} \leftarrow \texttt{oldrank} + 1$,
    $\texttt{finished} \leftarrow \textrm{true}$
    \item
    For each index $i$ for which $\texttt{rankarray}[i] = \texttt{oldrank}$, do:
      \begin{enumerate}
      \item
      Let $X$ be the unflattening of the decoding of $i$.
      \item
      Set $Y \leftarrow X + E$:
      this amounts to changing the rightmost bit of the flattening of $X$ from 0 to 1 or from 1 to 0.
      \item
      Let $j$ be the encoding of the flattening of $Y$.
      Thus $j = i+1$ if $i$ is even, and $j = i-1$ if $i$ is odd.
      \item
      If $\texttt{rankarray}[j] = 0$, then $Y$ has rank $\texttt{oldrank}+1$.
      In this case:
        \begin{itemize}
        \item
        Use \texttt{linkarray} to store $\texttt{oldrank}+1$ in every entry
        of \texttt{rankarray} corresponding to the tensors in the orbit of $Y$.
        \item
        $\texttt{finished} \leftarrow \textrm{false}$
        \end{itemize}
      \end{enumerate}
    \end{enumerate}
  \end{enumerate}
The iteration terminates when every entry of \texttt{rankarray} contains a positive integer,
which is the rank of the corresponding (nonzero) tensor.

To reduce the number of orbits, we consider the larger group
  \[
  G = \big( GL_3(\mathbb{F}_2) \times GL_3(\mathbb{F}_2) \times GL_3(\mathbb{F}_2) \big) \rtimes S_3,
  \]
where the symmetric group $S_3$ permutes the three directions.
We first compute the small orbits obtained by the action of $GL_3(\mathbb{F}_2) \times GL_3(\mathbb{F}_2) \times GL_3(\mathbb{F}_2)$
and then apply the permutations to determine which small orbits combine to make a single large orbit.
Given the canonical form $X$ of a small orbit $\mathcal{O}$ with index number $i$,
we apply the elements of $S_3$ to obtain tensors $X_1 = X, \dots, X_6$.
We then use the Maple arrays, which we have already computed, to find the index numbers
$i_1 = i, \dots, i_6$ of the small orbits containing these tensors.
We conclude that the union $\mathcal{O}_{i_1} \cup \cdots \cup \mathcal{O}_{i_6}$
is a large orbit for the action of $G$.
The canonical form for this large orbit is the smallest (in lex order)
of the canonical forms of $\mathcal{O}_{i_1}, \dots, \mathcal{O}_{i_6}$.
There are 115 (nonzero) small orbits and 55 (nonzero) large orbits:
\[
\begin{array}{lrrrrrrr}
\text{rank}
& 0 & 1 & 2 & 3 & 4 & 5 & 6
\\
\text{$\#$ small}
& 1 & 1 & 4 & 18 & 44 & 45 & 3
\\
\text{$\#$ large}
& 1 & 1 & 2 & 8 & 18 & 23 & 3
\\
\text{$\#$ tensors}
& 1 & 343 & 43218 & 2372286 & 47506872 & 83670048 & 624960
\\
\text{percent}
& 0.0000
& 0.0003
& 0.0322
& 1.7675
& 35.3954
& 62.3390
& 0.4656
\end{array}
\]
For the large orbit sizes and canonical forms, see Table \ref{table333}.
This computation took just under 282 minutes with Maple 16 on a
Lenovo ThinkCentre M91p Tower 7052A8U i7-2600 CPU (Quad Core 3.40/3.80GHz)
using Windows 7 Professional 64-bit
with 16 gigabytes of RAM.

For similar results for other tensor formats over $\mathbb{F}_2$, see \cite{BH,BS}.

\newcommand{\nn}{\!\!\!\!}
\newcommand{\sss}{-2pt}

\begin{table} \small
\begin{tabular}{rrrccccccccccccccccccccccccccc}
$\#$ & rank & size & \multicolumn{27}{c}{canonical form} \\ \midrule
   1 & 1 &      343 & .&\nn .&\nn .&\nn .&\nn .&\nn .&\nn .&\nn .&\nn .&\nn .&\nn .&\nn .&\nn .&\nn .&\nn .&\nn .&\nn .&\nn .&\nn .&\nn .&\nn .&\nn .&\nn .&\nn .&\nn .&\nn .&\nn 1 \\ \midrule
   2 & 2 &     6174 & .&\nn .&\nn .&\nn .&\nn .&\nn .&\nn .&\nn .&\nn .&\nn .&\nn .&\nn .&\nn .&\nn .&\nn .&\nn .&\nn .&\nn .&\nn .&\nn .&\nn .&\nn .&\nn .&\nn 1&\nn .&\nn 1&\nn . \\[\sss]
   3 & 2 &    37044 & .&\nn .&\nn .&\nn .&\nn .&\nn .&\nn .&\nn .&\nn .&\nn .&\nn .&\nn .&\nn .&\nn .&\nn .&\nn .&\nn .&\nn 1&\nn .&\nn .&\nn .&\nn .&\nn 1&\nn .&\nn .&\nn .&\nn . \\ \midrule
   4 & 3 &     3528 & .&\nn .&\nn .&\nn .&\nn .&\nn .&\nn .&\nn .&\nn .&\nn .&\nn .&\nn .&\nn .&\nn .&\nn .&\nn .&\nn .&\nn .&\nn .&\nn .&\nn 1&\nn .&\nn 1&\nn .&\nn 1&\nn .&\nn . \\[\sss]
   5 & 3 &     4116 & .&\nn .&\nn .&\nn .&\nn .&\nn .&\nn .&\nn .&\nn .&\nn .&\nn .&\nn .&\nn .&\nn .&\nn 1&\nn .&\nn 1&\nn .&\nn .&\nn .&\nn .&\nn .&\nn 1&\nn .&\nn .&\nn 1&\nn 1 \\[\sss]
   6 & 3 &    18522 & .&\nn .&\nn .&\nn .&\nn .&\nn .&\nn .&\nn .&\nn .&\nn .&\nn .&\nn .&\nn .&\nn .&\nn .&\nn .&\nn .&\nn 1&\nn .&\nn .&\nn .&\nn .&\nn .&\nn 1&\nn .&\nn 1&\nn . \\[\sss]
   7 & 3 &   148176 & .&\nn .&\nn .&\nn .&\nn .&\nn .&\nn .&\nn .&\nn .&\nn .&\nn .&\nn .&\nn .&\nn .&\nn 1&\nn .&\nn 1&\nn .&\nn .&\nn .&\nn .&\nn .&\nn 1&\nn .&\nn 1&\nn .&\nn . \\[\sss]
   8 & 3 &   222264 & .&\nn .&\nn .&\nn .&\nn .&\nn .&\nn .&\nn .&\nn .&\nn .&\nn .&\nn .&\nn .&\nn .&\nn .&\nn .&\nn .&\nn 1&\nn .&\nn .&\nn .&\nn .&\nn 1&\nn .&\nn 1&\nn .&\nn . \\[\sss]
   9 & 3 &   592704 & .&\nn .&\nn .&\nn .&\nn .&\nn .&\nn .&\nn .&\nn .&\nn .&\nn .&\nn .&\nn .&\nn .&\nn .&\nn .&\nn .&\nn 1&\nn .&\nn 1&\nn .&\nn 1&\nn .&\nn .&\nn .&\nn .&\nn . \\[\sss]
  10 & 3 &   592704 & .&\nn .&\nn .&\nn .&\nn .&\nn .&\nn .&\nn .&\nn .&\nn .&\nn .&\nn .&\nn .&\nn .&\nn 1&\nn .&\nn 1&\nn .&\nn 1&\nn .&\nn .&\nn .&\nn .&\nn .&\nn .&\nn 1&\nn . \\[\sss]
  11 & 3 &   790272 & .&\nn .&\nn .&\nn .&\nn .&\nn .&\nn .&\nn .&\nn 1&\nn .&\nn .&\nn .&\nn .&\nn 1&\nn .&\nn .&\nn .&\nn .&\nn 1&\nn .&\nn .&\nn .&\nn .&\nn .&\nn .&\nn .&\nn . \\ \midrule
  12 & 4 &   148176 & .&\nn .&\nn .&\nn .&\nn .&\nn .&\nn .&\nn .&\nn .&\nn .&\nn .&\nn .&\nn .&\nn .&\nn 1&\nn .&\nn 1&\nn .&\nn .&\nn .&\nn 1&\nn .&\nn .&\nn .&\nn 1&\nn .&\nn . \\[\sss]
  13 & 4 &   197568 & .&\nn .&\nn .&\nn .&\nn .&\nn 1&\nn .&\nn 1&\nn .&\nn .&\nn .&\nn 1&\nn .&\nn .&\nn .&\nn 1&\nn .&\nn .&\nn .&\nn 1&\nn .&\nn 1&\nn 1&\nn .&\nn .&\nn 1&\nn . \\[\sss]
  14 & 4 &   222264 & .&\nn .&\nn .&\nn .&\nn .&\nn .&\nn .&\nn .&\nn .&\nn .&\nn .&\nn .&\nn .&\nn .&\nn .&\nn .&\nn .&\nn 1&\nn .&\nn .&\nn 1&\nn .&\nn 1&\nn .&\nn 1&\nn .&\nn . \\[\sss]
  15 & 4 &   263424 & .&\nn .&\nn .&\nn .&\nn .&\nn .&\nn .&\nn .&\nn 1&\nn .&\nn 1&\nn .&\nn 1&\nn .&\nn .&\nn .&\nn .&\nn .&\nn 1&\nn .&\nn .&\nn 1&\nn 1&\nn .&\nn .&\nn .&\nn . \\[\sss]
  16 & 4 &   444528 & .&\nn .&\nn .&\nn .&\nn .&\nn .&\nn .&\nn .&\nn .&\nn .&\nn .&\nn .&\nn .&\nn .&\nn 1&\nn .&\nn 1&\nn .&\nn .&\nn .&\nn 1&\nn .&\nn 1&\nn .&\nn 1&\nn .&\nn . \\[\sss]
  17 & 4 &   592704 & .&\nn .&\nn .&\nn .&\nn .&\nn .&\nn .&\nn .&\nn .&\nn .&\nn .&\nn .&\nn .&\nn .&\nn 1&\nn .&\nn 1&\nn .&\nn 1&\nn .&\nn .&\nn .&\nn 1&\nn .&\nn .&\nn 1&\nn 1 \\[\sss]
  18 & 4 &  1185408 & .&\nn .&\nn .&\nn .&\nn .&\nn .&\nn .&\nn .&\nn 1&\nn .&\nn .&\nn .&\nn .&\nn .&\nn 1&\nn .&\nn 1&\nn .&\nn 1&\nn .&\nn .&\nn .&\nn .&\nn .&\nn .&\nn .&\nn . \\[\sss]
  19 & 4 &  1778112 & .&\nn .&\nn .&\nn .&\nn .&\nn .&\nn .&\nn .&\nn .&\nn .&\nn .&\nn .&\nn .&\nn .&\nn 1&\nn .&\nn 1&\nn .&\nn .&\nn .&\nn 1&\nn 1&\nn .&\nn .&\nn .&\nn .&\nn . \\[\sss]
  20 & 4 &  1778112 & .&\nn .&\nn .&\nn .&\nn .&\nn .&\nn .&\nn .&\nn 1&\nn .&\nn .&\nn .&\nn .&\nn .&\nn .&\nn .&\nn 1&\nn .&\nn .&\nn .&\nn 1&\nn 1&\nn .&\nn .&\nn .&\nn .&\nn . \\[\sss]
  21 & 4 &  1778112 & .&\nn .&\nn .&\nn .&\nn .&\nn .&\nn .&\nn .&\nn 1&\nn .&\nn .&\nn .&\nn .&\nn 1&\nn .&\nn .&\nn .&\nn .&\nn .&\nn 1&\nn 1&\nn 1&\nn .&\nn .&\nn 1&\nn .&\nn . \\[\sss]
  22 & 4 &  2370816 & .&\nn .&\nn .&\nn .&\nn .&\nn .&\nn .&\nn .&\nn 1&\nn .&\nn 1&\nn .&\nn 1&\nn .&\nn .&\nn .&\nn .&\nn .&\nn 1&\nn .&\nn .&\nn 1&\nn 1&\nn .&\nn .&\nn 1&\nn . \\[\sss]
  23 & 4 &  2370816 & .&\nn .&\nn .&\nn .&\nn .&\nn .&\nn .&\nn .&\nn 1&\nn .&\nn 1&\nn .&\nn 1&\nn .&\nn .&\nn .&\nn .&\nn .&\nn 1&\nn .&\nn .&\nn 1&\nn 1&\nn 1&\nn .&\nn 1&\nn . \\[\sss]
  24 & 4 &  3556224 & .&\nn .&\nn .&\nn .&\nn .&\nn .&\nn .&\nn .&\nn 1&\nn .&\nn .&\nn .&\nn .&\nn .&\nn 1&\nn .&\nn 1&\nn .&\nn 1&\nn .&\nn .&\nn .&\nn .&\nn .&\nn .&\nn 1&\nn . \\[\sss]
  25 & 4 &  4741632 & .&\nn .&\nn .&\nn .&\nn .&\nn .&\nn .&\nn .&\nn 1&\nn .&\nn .&\nn .&\nn .&\nn 1&\nn .&\nn .&\nn .&\nn .&\nn 1&\nn .&\nn .&\nn .&\nn .&\nn .&\nn .&\nn 1&\nn . \\[\sss]
  26 & 4 &  4741632 & .&\nn .&\nn .&\nn .&\nn .&\nn .&\nn .&\nn .&\nn 1&\nn .&\nn .&\nn .&\nn .&\nn 1&\nn .&\nn 1&\nn .&\nn .&\nn 1&\nn .&\nn .&\nn .&\nn .&\nn 1&\nn .&\nn .&\nn . \\[\sss]
  27 & 4 &  7112448 & .&\nn .&\nn .&\nn .&\nn .&\nn .&\nn .&\nn .&\nn 1&\nn .&\nn .&\nn .&\nn .&\nn .&\nn 1&\nn .&\nn 1&\nn .&\nn 1&\nn .&\nn .&\nn .&\nn 1&\nn .&\nn .&\nn .&\nn . \\[\sss]
  28 & 4 &  7112448 & .&\nn .&\nn .&\nn .&\nn .&\nn .&\nn .&\nn .&\nn 1&\nn .&\nn .&\nn .&\nn .&\nn 1&\nn .&\nn .&\nn .&\nn .&\nn 1&\nn .&\nn .&\nn .&\nn .&\nn 1&\nn .&\nn 1&\nn . \\[\sss]
  29 & 4 &  7112448 & .&\nn .&\nn .&\nn .&\nn .&\nn .&\nn .&\nn .&\nn 1&\nn .&\nn .&\nn .&\nn .&\nn 1&\nn .&\nn 1&\nn .&\nn .&\nn .&\nn 1&\nn 1&\nn 1&\nn .&\nn .&\nn .&\nn .&\nn . \\ \midrule
  30 & 5 &    28224 & .&\nn .&\nn .&\nn .&\nn .&\nn 1&\nn .&\nn 1&\nn .&\nn .&\nn .&\nn 1&\nn .&\nn .&\nn .&\nn 1&\nn .&\nn .&\nn .&\nn 1&\nn .&\nn 1&\nn .&\nn .&\nn .&\nn .&\nn . \\[\sss]
  31 & 5 &   148176 & .&\nn .&\nn .&\nn .&\nn .&\nn .&\nn .&\nn .&\nn 1&\nn .&\nn .&\nn .&\nn .&\nn .&\nn .&\nn .&\nn 1&\nn .&\nn .&\nn .&\nn 1&\nn .&\nn 1&\nn .&\nn 1&\nn .&\nn . \\[\sss]
  32 & 5 &   148176 & .&\nn .&\nn .&\nn .&\nn .&\nn .&\nn .&\nn .&\nn 1&\nn .&\nn .&\nn .&\nn .&\nn .&\nn 1&\nn .&\nn 1&\nn .&\nn .&\nn .&\nn 1&\nn .&\nn 1&\nn .&\nn 1&\nn .&\nn . \\[\sss]
  33 & 5 &   169344 & .&\nn .&\nn .&\nn .&\nn .&\nn .&\nn .&\nn .&\nn .&\nn .&\nn .&\nn 1&\nn .&\nn 1&\nn .&\nn 1&\nn .&\nn .&\nn .&\nn 1&\nn .&\nn 1&\nn .&\nn .&\nn .&\nn 1&\nn 1 \\[\sss]
  34 & 5 &   592704 & .&\nn .&\nn .&\nn .&\nn .&\nn 1&\nn .&\nn 1&\nn .&\nn .&\nn .&\nn 1&\nn .&\nn .&\nn .&\nn 1&\nn .&\nn .&\nn .&\nn 1&\nn .&\nn 1&\nn 1&\nn .&\nn .&\nn 1&\nn 1 \\[\sss]
  35 & 5 &  1185408 & .&\nn .&\nn .&\nn .&\nn .&\nn 1&\nn .&\nn 1&\nn .&\nn .&\nn .&\nn 1&\nn .&\nn .&\nn .&\nn 1&\nn .&\nn .&\nn .&\nn 1&\nn .&\nn 1&\nn .&\nn .&\nn .&\nn 1&\nn . \\[\sss]
  36 & 5 &  1580544 & .&\nn .&\nn .&\nn .&\nn .&\nn 1&\nn .&\nn 1&\nn .&\nn .&\nn .&\nn 1&\nn .&\nn .&\nn .&\nn 1&\nn .&\nn .&\nn 1&\nn .&\nn .&\nn 1&\nn 1&\nn .&\nn .&\nn .&\nn . \\[\sss]
  37 & 5 &  1580544 & .&\nn .&\nn .&\nn .&\nn .&\nn 1&\nn .&\nn 1&\nn .&\nn .&\nn .&\nn 1&\nn .&\nn .&\nn .&\nn 1&\nn .&\nn .&\nn 1&\nn .&\nn .&\nn 1&\nn 1&\nn .&\nn .&\nn .&\nn 1 \\[\sss]
  38 & 5 &  1778112 & .&\nn .&\nn .&\nn .&\nn .&\nn .&\nn .&\nn .&\nn 1&\nn .&\nn .&\nn .&\nn .&\nn .&\nn 1&\nn .&\nn 1&\nn .&\nn .&\nn .&\nn 1&\nn 1&\nn .&\nn .&\nn .&\nn .&\nn . \\[\sss]
  39 & 5 &  1778112 & .&\nn .&\nn .&\nn .&\nn .&\nn .&\nn .&\nn .&\nn 1&\nn .&\nn .&\nn .&\nn .&\nn 1&\nn .&\nn 1&\nn .&\nn .&\nn .&\nn .&\nn 1&\nn 1&\nn .&\nn .&\nn 1&\nn 1&\nn . \\[\sss]
  40 & 5 &  2370816 & .&\nn .&\nn .&\nn .&\nn .&\nn 1&\nn .&\nn 1&\nn .&\nn .&\nn .&\nn 1&\nn .&\nn .&\nn .&\nn 1&\nn .&\nn .&\nn .&\nn 1&\nn .&\nn 1&\nn 1&\nn .&\nn 1&\nn .&\nn . \\[\sss]
  41 & 5 &  2370816 & .&\nn .&\nn .&\nn .&\nn .&\nn 1&\nn .&\nn 1&\nn .&\nn .&\nn .&\nn 1&\nn .&\nn 1&\nn .&\nn 1&\nn .&\nn .&\nn 1&\nn .&\nn .&\nn 1&\nn .&\nn .&\nn .&\nn .&\nn 1 \\[\sss]
  42 & 5 &  2370816 & .&\nn .&\nn .&\nn .&\nn .&\nn 1&\nn .&\nn 1&\nn .&\nn .&\nn .&\nn 1&\nn .&\nn 1&\nn .&\nn 1&\nn .&\nn .&\nn 1&\nn .&\nn .&\nn 1&\nn .&\nn .&\nn 1&\nn .&\nn 1 \\[\sss]
  43 & 5 &  3556224 & .&\nn .&\nn .&\nn .&\nn .&\nn .&\nn .&\nn .&\nn 1&\nn .&\nn .&\nn .&\nn .&\nn .&\nn 1&\nn .&\nn 1&\nn .&\nn .&\nn 1&\nn .&\nn 1&\nn .&\nn .&\nn .&\nn .&\nn . \\[\sss]
  44 & 5 &  4741632 & .&\nn .&\nn .&\nn .&\nn .&\nn .&\nn .&\nn .&\nn 1&\nn .&\nn 1&\nn .&\nn 1&\nn .&\nn .&\nn .&\nn .&\nn .&\nn 1&\nn .&\nn .&\nn 1&\nn 1&\nn 1&\nn 1&\nn .&\nn . \\[\sss]
  45 & 5 &  4741632 & .&\nn .&\nn .&\nn .&\nn .&\nn 1&\nn .&\nn 1&\nn .&\nn .&\nn .&\nn 1&\nn .&\nn 1&\nn .&\nn 1&\nn .&\nn .&\nn .&\nn 1&\nn .&\nn 1&\nn .&\nn .&\nn 1&\nn .&\nn . \\[\sss]
  46 & 5 &  4741632 & .&\nn .&\nn .&\nn .&\nn .&\nn 1&\nn .&\nn 1&\nn .&\nn .&\nn .&\nn 1&\nn .&\nn 1&\nn .&\nn 1&\nn .&\nn .&\nn 1&\nn .&\nn .&\nn .&\nn .&\nn .&\nn .&\nn 1&\nn . \\[\sss]
  47 & 5 &  4741632 & .&\nn .&\nn .&\nn .&\nn .&\nn 1&\nn .&\nn 1&\nn .&\nn .&\nn .&\nn 1&\nn .&\nn 1&\nn .&\nn 1&\nn .&\nn .&\nn 1&\nn .&\nn .&\nn .&\nn .&\nn .&\nn .&\nn 1&\nn 1 \\[\sss]
  48 & 5 &  4741632 & .&\nn .&\nn .&\nn .&\nn .&\nn 1&\nn .&\nn 1&\nn .&\nn .&\nn .&\nn 1&\nn .&\nn 1&\nn .&\nn 1&\nn .&\nn .&\nn 1&\nn .&\nn .&\nn 1&\nn .&\nn .&\nn .&\nn 1&\nn . \\[\sss]
  49 & 5 &  4741632 & .&\nn .&\nn .&\nn .&\nn .&\nn 1&\nn .&\nn 1&\nn .&\nn .&\nn .&\nn 1&\nn .&\nn 1&\nn .&\nn 1&\nn .&\nn .&\nn 1&\nn .&\nn .&\nn 1&\nn .&\nn .&\nn 1&\nn 1&\nn . \\[\sss]
  50 & 5 &  7112448 & .&\nn .&\nn .&\nn .&\nn .&\nn .&\nn .&\nn .&\nn 1&\nn .&\nn .&\nn 1&\nn .&\nn 1&\nn .&\nn 1&\nn .&\nn .&\nn 1&\nn .&\nn .&\nn .&\nn .&\nn 1&\nn .&\nn 1&\nn . \\[\sss]
  51 & 5 & 14224896 & .&\nn .&\nn .&\nn .&\nn .&\nn .&\nn .&\nn .&\nn 1&\nn .&\nn .&\nn .&\nn .&\nn 1&\nn .&\nn 1&\nn .&\nn .&\nn 1&\nn .&\nn .&\nn .&\nn .&\nn 1&\nn .&\nn 1&\nn . \\[\sss]
  52 & 5 & 14224896 & .&\nn .&\nn .&\nn .&\nn .&\nn .&\nn .&\nn .&\nn 1&\nn .&\nn .&\nn 1&\nn .&\nn 1&\nn .&\nn 1&\nn .&\nn .&\nn .&\nn 1&\nn .&\nn 1&\nn .&\nn .&\nn .&\nn .&\nn . \\ \midrule
  53 & 6 &    32256 & .&\nn .&\nn 1&\nn .&\nn 1&\nn .&\nn 1&\nn .&\nn .&\nn .&\nn 1&\nn .&\nn 1&\nn .&\nn .&\nn .&\nn 1&\nn 1&\nn 1&\nn .&\nn .&\nn .&\nn 1&\nn 1&\nn 1&\nn 1&\nn . \\[\sss]
  54 & 6 &   197568 & .&\nn .&\nn .&\nn .&\nn .&\nn 1&\nn .&\nn 1&\nn .&\nn .&\nn .&\nn 1&\nn .&\nn .&\nn .&\nn 1&\nn .&\nn .&\nn .&\nn 1&\nn .&\nn 1&\nn .&\nn .&\nn .&\nn .&\nn 1 \\[\sss]
  55 & 6 &   395136 & .&\nn .&\nn .&\nn .&\nn .&\nn 1&\nn .&\nn 1&\nn .&\nn .&\nn .&\nn 1&\nn .&\nn 1&\nn .&\nn 1&\nn .&\nn .&\nn .&\nn 1&\nn .&\nn 1&\nn .&\nn .&\nn .&\nn 1&\nn 1 \\ \midrule
\end{tabular}
\caption{Large orbits of $3 \times 3 \times 3$ tensors over $\mathbb{F}_2$}
\label{table333}
\end{table}


\section{Acknowledgements}

The first author was partially supported by a Discovery Grant from NSERC.
We thank J. M. F. ten Berge for sending us the unpublished notes by Rocci \cite{Rocci}.


\begin{thebibliography}{9}
\bibitem{Rocci}
\textsc{R. Rocci}:
A $3 \times 3 \times 3$ array has maximum rank 5.
(Microsoft Word document, no date.)
\end{thebibliography}
\end{document}